\documentclass[11pt]{article}
\usepackage{amsmath,amssymb,amsthm,comment}
\usepackage{mathtools}
\usepackage{color}
\usepackage[nohead,margin=1.1in]{geometry}
\usepackage[ruled,vlined]{algorithm2e}

\usepackage{booktabs}
\usepackage{threeparttable}
\usepackage{graphicx,psfrag,epstopdf}
\usepackage{url}
\graphicspath{{./pics/}}
\graphicspath{{./fig/}}
\usepackage{tcolorbox}
\allowdisplaybreaks
\theoremstyle{plain}
\newtheorem{theorem}{Theorem}[section]
\newtheorem{remark}{Remark}[section]
\newtheorem{definition}{Definition}[section]
\newtheorem{lemma}{Lemma}[section]
\newtheorem{assumption}{Assumption}[section]
\newtheorem{prop}{Proposition}[section]

\newtheorem{corollary}{Corollary}[section]
\numberwithin{equation}{section}
\def\E{\mathbb{E}}

\newcommand{\Prob}{\mathbb{P}}

\begin{document}
\title{On the convergence of doubly stochastic block gradient descent method for linear inverse problems
\thanks{Y. Xia is supported by the National Natural Science Foundation of China (No. 12401556).}}

\author{Yuxin Xia\thanks{School of Mathematics, Harbin Institute of Technology, Harbin 150001, Heilongjiang Province, China (\texttt{yuxinxia@hit.edu.cn}) } 
\and Zehui Zhou\thanks{Department of Mathematics, Hong Kong Baptist University,
Kowloon Tong, Hong Kong, China (\texttt{zehuizhou@hkbu.edu.hk}). Author to whom any correspondence should be addressed.}}

\date{}
\maketitle

\begin{abstract}
Stochastic gradient descent (SGD) and randomized block coordinate descent (RBCD) belong to two fundamentally different classes of stochastic iterative regularization methods, and are promising and popular for solving large-scale inverse problems, due to their excellent scalability with respect to data size and signal size, respectively. 
In this work, we propose a doubly stochastic block gradient descent (DSBGD) method with early stopping, by integrating the updating strategies of both SGD and RBCD, for solving linear inverse problems in infinite-dimensional Hilbert spaces.
At each iteration, the method randomly selects one data pair to obtain a stochastic gradient estimate, and then uses it to update one randomly selected block coordinate.
We prove the regularizing property of DSBGD under \textit{a priori} stopping rules and then derive its convergence rates with the additional source condition and a condition on the forward operator.
To the best of our knowledge, these are the first convergence rate results of randomized coordinate descent-type iterative methods for inverse problems from the perspective of regularization theory.
We also prove that the DSBGD iteration can be stopped by the classical discrepancy principle in finitely many steps with high probability, and that the mean squared error decreases monotonically before stopping. 
Several numerical experiments are provided to complement the theoretical analysis.


\end{abstract}

\section{Introduction}

In this work, we develop and analyze stochastic iterative regularization methods for solving linear ill-posed inverse problems in Hilbert spaces:
\begin{equation}\label{eqn:lininv}
A x=y^\dag,
\end{equation}
where the operator $A: X=X_1 \times \cdots \times  X_{n_x} \rightarrow Y =Y_1\times \cdots \times Y_{n_y}$, representing the data formation mechanism, is given by $$A x := (A_1 x,\cdots,A_{n_y} x)^\top, \quad \forall x\in X,$$ with bounded linear operators $A_i: X \rightarrow Y_i$ and the superscript $\top $ denoting the vector transpose. 
Each operator $A_i$ is of the form $$A_{i}x=\sum_{j=1}^{n_x} A_{i,j}x_j$$ 
with bounded linear operators $A_{i,j}$ between Hilbert spaces $X_j$ and $Y_i$ equipped with norms $\|\cdot\|_{X_i}$ and $\|\cdot\|_{Y_i}$, respectively. 
We also define the bounded linear operator $A_{\cdot,j}: X_j \rightarrow Y$ such that $A_{\cdot,j}x_j:=(A_{1,j}x_j,\cdots,A_{n_y,j}x_j)^\top$.
The element $y^\dag=(y^\dag_1,\cdots,y^\dag_{n_y})^\top = A x^\dag \in Y$ denotes the exact data with the norm $\|y^\dag\|_Y:=\sqrt{\sum_{i=1}^{{n_y}}\|y^\dag_i\|_{Y_i}^2}$ and $x^\dag=(x^\dag_1,\cdots,x^\dag_{n_x})^\top$ being the minimum-norm solution relative to the initial guess $x_1$ (cf.\;\eqref{eqn:min-norm}),
while $x=(x_1,\cdots,x_{n_x})^\top\in X$ denotes the unknown signal of interest with the norm $\|x\|_X:=\sqrt{\sum_{j=1}^{n_x}\|x_j\|^2_{X_j}}$. 
In practice, we only have access to a noisy version of $y^\dag$, given by
$$y^\delta =(y^\delta_1,\cdots,y^\delta_{n_y})^\top =y^\dag +\xi,$$
where $\xi = (\xi_1,\cdots,\xi_{n_y})^\top\in Y$ denotes the noise in the data with a noise level $\delta =\|\xi\|_Y$.
Linear inverse problems of the form \eqref{eqn:lininv} arise in a wide range of practical applications, particularly in medical imaging including computed tomography (CT) \cite{HermanLentLutz:1978,GaoBlumensath:2018} and photoacoustic tomography (PAT) \cite{PoudelLouAnastasio:2019,XiaWangHan:2022}.
Recent advances in data acquisition have made large-scale datasets, such as dynamic, multi-spectral, multi-frequency and multi-energy data in CT and PAT, increasingly common. 
However, the potential for high-volume data to improve reconstruction accuracy is counterbalanced by a significant computational bottleneck, i.e., massive datasets combined with costly forward models will make classical reconstruction techniques \cite{EnglHankeNeubauer:1996}, such as variational regularization and deterministic iterative methods, computationally prohibitive.
Meanwhile, the corresponding inverse problems often involve high-dimensional or even infinite-dimensional unknown signals with inherent block structures, such as those arising from different spatial subdomains or material components. 
These two sources of large-scale complexity motivate scalable regularization methods in both data and signal sizes (i.e., $n_y$ and $n_x$).

In recent years, stochastic iterative regularization methods, including stochastic gradient descent (SGD) \cite{RobbinsMonro:1951,JinLu:2019,JahnJin:2020,JinZhouZou:2021siuq,ZhouDSGD:2024} and randomized block coordinate descent (RBCD) \cite{Nesterov：2012,JinLiu:2025}, have gained much interest in the inverse problems community.
We refer to recent surveys \cite{EhrhardtKereta:2025,JinXiaZhou:2025} for an overview.
A notable strength of this class of methods is that the computational cost per iteration is independent of the data size $n_y$ or the signal size $n_x$, which enables direct scalability to massive datasets and high-dimensional signals, and is therefore especially appealing in the era of big data.
However, despite their computational appeal, the mathematical analysis of these two types of stochastic iterative methods for inverse problems in the lens of regularization theory has not been fully explored, especially their combinations.

Both SGD and RBCD can be interpreted as randomized versions of the classical Landweber method \cite{Landweber:1951}.
To illustrate the idea, we consider the following ill-posed optimization problem 
\begin{equation*}
\min_{x\in X} J(x),
\end{equation*}
where the objective function $J(x)$ is given by
\begin{align*}
J (x) &= \frac{1}{2 n_y}\|A x-y^\delta\|_Y^2 = \frac{1}{n_y}\sum_{i=1}^{n_y} f_i(x),\\
\mbox{with} \quad f_i(x)&=\frac{1}{2} \|A_ix-y_i^\delta\|_{Y_i}^2=\frac{1}{2} \bigg\|\sum_{j=1}^{n_x} A_{i,j}x_j-y_i^\delta\bigg\|_{Y_i}^2.
\end{align*}
Given an initial guess $x_1^\delta\equiv x_1\in X$ and the stepsize schedule $\{\eta_k\}_{k\geq 1}\subset (0,\infty)$, SGD reads
\begin{equation*}
x_{k+1}^\delta = x_k^\delta - \eta_k f'_{i_k}({ x_k^\delta})
, \quad k=1,2,\cdots,
\end{equation*}
while RBCD updates the iterate $x_{k+1}^\delta= (x_{k+1,1}^\delta, \cdots, x_{k+1, n_x}^\delta)^\top$ by
\begin{equation*}
x_{k+1,j}^\delta = 
\begin{cases}
x_{k,j_k}^\delta - \eta_k \nabla_{j_k} J(x_k^\delta), & \mbox{when} \;j =  j_k, \\
x_{k,j}^\delta, & \mbox{otherwise}, 
\end{cases}\quad k=1,2,\cdots,
\end{equation*}
where $\nabla_{j_k} J(x):=\nabla_{x_{j_k}} J(x)$ denotes the partial gradient of $J$ with respect to $x_{j_k}$, the indices $i_k$ and $j_k$ are sampled uniformly at random from the set $\{1,\ldots,n_y\}$ and $\{1,\ldots,n_x\}$, respectively. 
At each iteration, SGD updates all components of $x_{k+1}^\delta$ using the gradient $f'_{i_k}$ derived from a single data pair $(A_{i_k},y_{i_k}^\delta)$ as the random gradient estimate, while RBCD only updates one component of $x_{k+1}^\delta$ using a partial gradient of the entire objective $J$, leaving the other components unchanged. 
These two types of stochastic iterative methods are complementary rather than competing. 
SGD improves data scalability, whereas RBCD enhances signal scalability. 
This motivates a doubly randomized method that samples both data and signal at each iteration for large-scale and structured inverse problems. 
Combining the updating strategies of both SGD and RBCD gives the idea of the doubly stochastic block gradient descent (DSBGD) method (named the double stochastic block method for well-posed optimization
problems in \cite{BaoGuoXingLi:2026}) that reads
\begin{equation}\label{eqn:DSBGD}
x_{k+1,j}^\delta = 
\begin{cases}
x_{k,j_k}^\delta - \eta_k \nabla_{j_k} f_{i_k}(x_k^\delta)=x_{k,j_k}^\delta - \eta_kA_{i_k,j_k}^*(A_{i_k}{x_k^\delta}-y_{i_k}^\delta)
, & \mbox{when} \;j =  j_k, \\
x_{k,j}^\delta, & \mbox{otherwise}, 
\end{cases}
\end{equation}
where $A_{i,j}^*$ denotes the adjoint of  $A_{i,j}$.
It is worth noting that during the numerical implementation, the term $A_{i_k}x_k^\delta-y_{i_k}^\delta=(Ax_k^\delta-y^\delta)_{i_k}$ in \eqref{eqn:DSBGD} is the $i_k$th component of the residual $r_{k}^\delta:=Ax_k^\delta-y^\delta$ and can be calculated in a recursive way with less computational cost than computing the term $A_{i_k}x_k^\delta-y_{i_k}^\delta$ directly:
\begin{align*}
r_{k+1}^\delta = r_{k}^\delta+A(x_{k+1}^\delta-x_{k}^\delta)
=r_{k}^\delta+A_{\cdot,j_k}(x_{k+1,j_k}^\delta-x_{k,j_k}^\delta).
\end{align*}
Beyond uniform sampling, other strategies for selecting the indices $i_k$ and $j_k$ can also be used, including random shuffling \cite{Ying:2018,Haochen:2019,SafranShamir:2020} and mini-batch sampling \cite{LuMathe2022}. 
In this work, we focus on a simple version of the DSBGD method defined in \eqref{eqn:DSBGD}, with an equivalent version given in Algorithm \ref{alg:DSBGD} for numerical implementation, and leave extensions to future work.

\medskip
\begin{algorithm}[H]
\SetAlgoLined
Set initial guess $x_1^\delta=x_1\in X$, initial residual $r_1^\delta=Ax_1-y^\delta$, and stepsize schedule $\{\eta_k\}_{k\geq 1}\subset (0,\infty)$\\
 \For{$k=1,2,\cdots$}{
Draw $i_k$ and $j_k$ i.i.d. uniformly from $\{1,\cdots,n_y\}$ and $\{1,\cdots,n_x\}$, respectively\\
Compute $g = \eta_k A_{i_k,j_k}^*r_{k,i_k}^\delta$\\
Update 
$\;\;
x_{k+1,j_k}^\delta = x_{k,j_k}^\delta - g, \quad  r_{k+1}^\delta 
=r_{k}^\delta-A_{\cdot,j_k} g \quad
\mbox{and}\quad x_{k+1,j}^\delta = x_{k,j}^\delta, \;\;\forall j\neq j_k
$
\\
Check the stopping criterion}
\caption{Doubly Stochastic Block Gradient Descent (DSBGD) for problem \eqref{eqn:lininv}.\label{alg:DSBGD}}
\end{algorithm}

\medskip

In this work, we contribute to the convergence analysis of the proposed DSBGD defined in \eqref{eqn:DSBGD} for a class of linear inverse problems of the form \eqref{eqn:lininv} from the perspective of regularization theory. 
Under suitable \textit{a priori} stopping rules, we establish the regularizing property of the proposed method; see Theorem \ref{thm:conv-exact}, which guarantees the convergence for exact data, and Theorem \ref{thm:conv-noisy}, which covers the noisy case.
Further, with the additional source condition and a condition on the forward operator $A$ (i.e., Assumptions \ref{ass:source} and \ref{ass:A}), we achieve the order-optimal convergence rates of the proposed method with polynomially decaying stepsize schedules for low-regularity solutions, which are comparable with that for the Landweber method \cite{HankeNeubauerScherzer:1995} and SGD \cite{JinLu:2019,JinZhouZou:2020}.
Moreover, we investigate the proposed DSBGD method under the discrepancy principle, an \textit{a posteriori} stopping rule, and prove both that it terminates in a finite number of iterations with high probability and that the mean squared iterate errors decrease monotonically before stopping; see Theorem \ref{thm:termination} and Proposition \ref{prop:decay_dp}.
When $n_y=1$, the proposed DSBGD method defined in \eqref{eqn:DSBGD} reduced to RBCD. 
The convergence of (R)BCD has been investigated in the works \cite{RabanserNeumannHaltmeier:2019,JinLiu:2025} for linear inverse problems. 
Specifically, the work \cite{RabanserNeumannHaltmeier:2019} studied the block coordinate descent (BCD) method with the index $j_k$ selected in a cyclic fashion. 
The authors proved the regularizing property of BCD when the operator $A$ has a particular tensor product form, a condition stronger than Assumption \ref{ass:A}. 
Then \cite{JinLiu:2025} generalized these results to RBCD with a general forward operator $A$ and further derived a strong convergence result for the specific operator $A$ considered in \cite{RabanserNeumannHaltmeier:2019}.
However, convergence rates of these methods remain largely missing. 
Moreover, the additional randomness introduced by sampling data at each iteration destroys the monotonicity of quantities related to the residual, posing significant challenges for the convergence analysis of DSBGD.
From a different perspective, DSBGD has also been investigated for well-posed optimization problems in \cite{BaoGuoXingLi:2026}.
However, these results are
not applicable to ill-posed problems, which differ fundamentally from well-posed ones.
Compared with the existing works \cite{RabanserNeumannHaltmeier:2019,JinLiu:2025,BaoGuoXingLi:2026},
we contribute to both the regularizing property and convergence rates in terms of the noise level $\delta$ of the proposed method for linear inverse problems of the form \eqref{eqn:lininv}.
To the best of our knowledge, this is the first convergence analysis for doubly stochastic iterative methods for inverse problems. Moreover, it provides the first convergence rate results for randomized coordinate descent-type iterative methods in this setting.

The rest of the work is organized as follows. 
First, the main results, i.e., the regularizing property and convergence rates under \textit{a priori} stopping rule in Theorems \ref{thm:conv-noisy}, \ref{thm:err-total} and Corollary \ref{cor:rate}, and the finite-iteration termination property under the discrepancy principle \eqref{eqn:discrepancy_1} in Theorem \ref{thm:termination}, along with relevant discussions, are presented in Section \ref{sec:main}. 
Then, detailed proofs and remarks on the regularizing and finite-iteration termination properties (Theorems \ref{thm:conv-noisy} and \ref{thm:termination}) and convergence rate analysis (Theorem \ref{thm:err-total}) are given in Sections \ref{sec:conv} and \ref{sec:rate}, respectively. 
Finally, several numerical experiments are presented in Section \ref{sec:num} to complement the theoretical analysis, indicating the feasibility and efficiency of DSBGD. 
For better readability, supplementary estimates and the lengthy technical proofs of several results are deferred to Appendix \ref{app:estimate}. 
Throughout, we denote the $(k-1)$th iterates for the exact data $y^\dag$ and the noisy data $y^\delta$ by $x_k$ and $x_k^\delta$ respectively.
We denote the inner product by $\langle\cdot,\cdot\rangle$ and suppress the subscripts in norms and inner products, as the underlying spaces are clear from the context.

\section{Main results and discussions}\label{sec:main}

In this section, we present the main results of the work. 
First, we define the reference solution $x^\dag$ as the unique minimum-norm solution relative to the initial guess $x_1$, given by
\begin{equation}\label{eqn:min-norm}
  x^\dag = \arg\min_{x\in X: A x = y^\dag} \|x-x_1\|.
\end{equation}
When analyzing the convergence behavior of an iterative method, the choice of stepsize schedules $\{\eta_k\}_{k\geq 1}$ plays a crucial role. 
Therefore, for analyzing the convergence of DSBGD, we introduce three classes of choices for the stepsize schedule $\{\eta_k\}_{k=1}^\infty$ in the following assumption.

\begin{assumption}\label{ass:stepsize}
Let $L:= \max_{1\leq i\leq n_y}\|A_i\|^2$, the stepsizes $\{\eta_k\}_{k\geq 1}$ satisfies one of the following conditions.
\begin{itemize}
\item[$\rm(i)$] $\eta_k =\mu_k L^{-1}$ with $0<\mu_k\leq c_0<2$ and $\sum_{k=1}^\infty \mu_k =\infty$.
\item[$\rm(ii)$] $\eta_k =\mu_k \|A_{i_k}\|^{-2}$ with $0<\mu_k\leq c_0<2$ and $\sum_{k=1}^\infty \mu_k =\infty$.
\item[$\rm(iii)$] $\eta_k=c_0L^{-1}k^{-\alpha}$ with $0<c_0< 2$.
\end{itemize}
\end{assumption}
The choice in Assumption \ref{ass:stepsize}(ii) represents a class of adaptive stepsizes that depends on the random sampling index $i_k$, and can be larger than that in (i). 
Under suitable \textit{a priori} stopping rules, both stepsize schedules in (i) and (ii) can turn the proposed DSBGD method into a regularization method; see Theorem \ref{thm:conv-noisy}.
Moreover, the stepsize in (i) is more general than that in (iii) used for deriving the convergence rates in Theorem \ref{thm:err-total}.
The latter is often referred to as a polynomially decaying stepsize schedule in the literature. 
Let $\mathcal{F}_k$ and $\mathcal{F}_k^X$ be the filtration generated by the random indices $\{i_1,\ldots,i_{k-1},j_1,\ldots,j_{k-1}\}$ and $\{j_1,\ldots,j_{k-1}\}$, respectively, $\mathcal{F}={\bigvee_{k=1}^\infty} \mathcal{F}_k$, $\mathcal{F}^X={\bigvee_{k=1}^\infty} \mathcal{F}_k^X$, and $(\Omega,\mathcal{F},\Prob)$ and $(\Omega,\mathcal{F}^X,\Prob)$ be the associated probability spaces. 
Then let $\E[\cdot]$ and $\E_X[\cdot]$ denote taking the expectation with respect to the filtration $\mathcal{F}$ and $\mathcal{F}^X$, respectively. 
The DSBGD iterate $x_k^\delta$ is random but measurable with respect to $\mathcal{F}_k$. 
Now, we state the main results on the regularizing property and convergence rates of DSBGD under \textit{a priori} stopping rules, as well as the finite-iteration termination property under the discrepancy principle (cf.\;\eqref{eqn:discrepancy_1}).

\begin{theorem}[Regularizing property under \textit{a priori} stopping rule]\label{thm:conv-noisy}
Let Assumption \ref{ass:stepsize}(i) or (ii) be fulfilled, and $x_k^\delta$ be the DSBGD iterate defined in \eqref{eqn:DSBGD}. 
If the stopping index $k(\delta)\geq1$ is chosen such that
\begin{equation*}
 \lim_{\delta\to0^+}k(\delta)=\infty \quad\mbox{and}\quad \lim_{\delta\to0^+}\delta^2\sum_{t=1}^{k(\delta)}\mu_t = 0,
\end{equation*} 
then there holds 
$\lim_{\delta\to0^+}\E[\|x_{k(\delta)}^\delta -x^\dag\|^2]=0.$
\end{theorem}

We also discuss the behavior of the DSBGD iterate when stopped by one of the most popular \textit{a posteriori} stopping rules, namely the discrepancy principle due to the Russian mathematician Vladimir Morozov \cite{Morozov:1966}. 
This rule chooses the stopping index $k(\delta)$ by
\begin{align}\label{eqn:discrepancy_1}
k(\delta)=\min\{k\geq 1:\;\|Ax_{k}^\delta-y^\delta\|\leq \tau \delta\},
\end{align}
with a fixed constant $\tau>1$.

\begin{theorem}[Finite-iteration termination property under the discrepancy principle \eqref{eqn:discrepancy_1}]\label{thm:termination}
Let Assumption \ref{ass:stepsize}(iii) be fulfilled with $\alpha\in(0,\frac23)$ and $c_0 \leq \alpha /2$, and the stopping index $k_*=k(\delta)$ be chosen by \eqref{eqn:discrepancy_1}. 
Then there hold
\begin{align*}
\mathbb{P}\Big(k(\delta)\leq \big[(c_{{\rm stop}}\delta^{-1})^{\frac{4}{2-3\alpha}}\big]+1\Big)\geq 1-4(\tau-1)^{-2}nc^*\big(c_{{\rm stop}}^{-\frac{4(1-\alpha)}{2-3\alpha}}+ c_{{\rm stop}}^{-\frac{2\alpha}{2-3\alpha}}\big)\delta^{\frac{2\alpha}{2-3\alpha}},
\end{align*}
and
\begin{align*}
\lim_{\delta\to 0^+}\mathbb{P}\Big(k(\delta)\leq \big[(c_{{\rm stop}}\delta^{-1})^{\frac{2}{1-\alpha}}\ln(e+\delta^{-1})\big]+1\Big)=1,
\end{align*}
with $[\cdot]$ taking the integral part of a real number, the constant $c_{{\rm stop}}=\frac{2\sqrt{nL}\|e_1^\delta\|}{\sqrt{2^\alpha e c_0}(\tau-1)}$, and the constant $c^*$ depending on $L$, $\alpha$ and $c_0$ but independent of $n$, $k$ and $\delta$.
\end{theorem}

\begin{remark}
Theorem \ref{thm:termination} indicates that the DSBGD iteration defined in \eqref{eqn:DSBGD} can be stopped by the discrepancy principle  \eqref{eqn:discrepancy_1} in finitely many steps with high probability, while Proposition \ref{prop:decay_dp} in Section \ref{sec:dp} guarantees the decay of the mean squared error $\E[\|x_{k}^\delta-x^\dag\|^2]$, i.e., a monotonically decreasing property, before stopping by rule \eqref{eqn:discrepancy_1}. 
\end{remark}

Next, we introduce two conditions, one on the exact solution $x^\dag$ and one on the forward operator $A$, which are crucial for deriving the convergence rates of DSBGD in Section \ref{sec:rate}.
The notation $\mathcal{N}(A)$ denotes the null space of $A$. 
\begin{assumption}[Source condition]\label{ass:source}
There exist some $\nu>0$ and $w\in \mathcal{N}(A)^\perp$ such that $x^\dag-x_1=B_X^\nu w$ and $\|w\|<\infty$, with $B_X=(n_x n_y)^{-1}A^*A$ and $\mathcal{N}(A)^\perp$ being the orthogonal complement of $\mathcal{N}(A)$.
\end{assumption}

\noindent Assumption \ref{ass:source} implicitly assumes a certain degree of regularity on the initial error $x^\dag-x_1$.
It is known as the source condition  \cite{EnglHankeNeubauer:1996} and is often used to establish the convergence rates of regularized solutions \cite{HankeNeubauerScherzer:1995,JinZhouZou:2020siopt,JinZhouZou:2021siuq,JinZhouZou:2022ip,JinZhou:2026}. 
Indeed, without a source condition, the convergence of regularized solutions may be arbitrarily slow \cite{EnglHankeNeubauer:1996}. 
Now, we provide the condition on $A$. We define $D_{j}^{n_x}:={\rm diag}(O_{X_1},\cdots,O_{X_{j-1}},I_{X_j},O_{X_{j-1}},\cdots,O_{X_{n_x}})$ with $O_{X_j}: X_j \rightarrow X_j$ and $I_{X_j}: X_j \rightarrow X_j$ being the zero operator and the identity operator, respectively.

\begin{assumption}\label{ass:A}
For any $j=1,\cdots,n_x$, there holds ${\rm range}(D_{j}^{n_x} A^*)\subseteq {\rm range}(A^*)$ or equivalently 
$D_{j}^{n_x} A^* A D_{j}^{n_x} \preceq c_j A^* A$,
with $C_A:=\sum_{j=1}^{n_x}c_j<\infty$.
\end{assumption}

\begin{remark}
We present a sufficient condition for Assumption \ref{ass:A}, cf.\;\eqref{eqn:condition_P}. 
Let $\{\lambda_t, \varphi_t\}_{t=1}^\infty$ denote the singular values and vectors of $B=(n_x  n_y)^{-1}A^*A$, i.e., $B(\cdot)=\sum_{t=1}^\infty \lambda_t\langle\varphi_t,\cdot\rangle\varphi_t$ and $\varphi_t=(\varphi_{t,1},\cdots,\varphi_{t,n_x})^\top\in X$ with $\varphi_{t,j}\in X_j$.
For any $1\leq j\leq n_x$, we define $P^{(j)}=(P^{(j)}_{t\ell})_{t,\ell\geq 1}$ with the $\{t\ell\}$-entry $P^{(j)}_{t\ell}:=\langle\varphi_{t,j} ,\varphi_{\ell,j}\rangle$. 
Let $\lambda=(\lambda_t)_{t\geq1}$ and $\Lambda :={\rm diag}(\lambda)$. 
Then it is easy to check that the inequality $D_{j}^{n_x} A^* A D_{j}^{n_x} \preceq c_j A^* A$ in Assumption \ref{ass:A} is equivalent to $(P^{(j)})^* \Lambda P^{(j)} \preceq c_j \Lambda$.
For any $a=(a_\ell)_{\ell\geq 1}$,  there hold
\begin{align*}
\langle (P^{(j)})^*  \Lambda P^{(j)}a,a\rangle&=  \langle  \Lambda P^{(j)}a,P^{(j)}a\rangle
= \sum_{t=1}^\infty \lambda_t\bigg(\sum_{\ell=1}^{\infty}a_\ell \langle \varphi_{t,j},\varphi_{\ell,j}\rangle\bigg)^2 \quad
\mbox{and}\quad \langle \Lambda a,a\rangle= 
\sum_{i=1}^\infty \lambda_i  a_i^2.
\end{align*}
Without loss of generality, we assume $\lambda_t>0$. 
Then by the Cauchy-Schwarz inequality, we have
\begin{align*}
\langle (P^{(j)})^*  \Lambda P^{(j)}a,a\rangle&=\sum_{t=1}^\infty \bigg(\sum_{\ell=1}^{\infty} \sqrt{\lambda_\ell} a_\ell \sqrt{\frac{\lambda_t}{\lambda_\ell}}\langle \varphi_{t,j},\varphi_{\ell,j}\rangle\bigg)^2
\leq \sum_{t=1}^\infty  \sum_{\ell=1}^{\infty} \frac{\lambda_t}{\lambda_\ell}|\langle \varphi_{t,j},\varphi_{\ell,j}\rangle|^2 \bigg(\sum_{i=1}^{\infty} \lambda_i a_i^2\bigg)\\
&=\sum_{t=1}^\infty  \sum_{\ell=1}^{\infty} \frac{\lambda_t}{\lambda_\ell}|\langle \varphi_{t,j},\varphi_{\ell,j}\rangle|^2 \langle \Lambda a,a\rangle,
\end{align*}
and thus
\begin{align*}
\sum_{j=1}^{n_x}\langle (P^{(j)})^*  \Lambda P^{(j)}a,a\rangle\leq
\sum_{\ell=1}^{\infty} \lambda_\ell^{-1}\sum_{t=1}^\infty   \lambda_t\sum_{j=1}^{n_x}|\langle \varphi_{t,j},\varphi_{\ell,j}\rangle|^2 \langle \Lambda a,a\rangle.
\end{align*}
Let $P=(P_{t\ell})_{t,\ell\geq 1}$ with the $\{t\ell\}$-entry $P_{t\ell}:=\sum_{j=1}^{n_x}|\langle\varphi_{t,j} ,\varphi_{\ell,j}\rangle|^2$, then
\begin{align*}
(P \lambda)_\ell= \sum_{t=1}^\infty P_{\ell t}\lambda_t
=\sum_{t=1}^\infty \lambda_t\sum_{j=1}^{n_x}|\langle\varphi_{\ell,j} ,\varphi_{t,j}\rangle|^2,
\end{align*}
which implies
\begin{align*}
\sum_{j=1}^{n_x}\langle (P^{(j)})^*  \Lambda P^{(j)}a,a\rangle\leq
\sum_{\ell=1}^{\infty} \lambda_\ell^{-1}  (P \lambda)_\ell\langle \Lambda a,a\rangle.
\end{align*}
If there exists a sequence of constants $\{C_\ell\}_{\ell\geq 1}$ such that 
\begin{equation}\label{eqn:condition_P}
(P \lambda)_\ell \leq C_\ell \lambda_\ell,\quad \forall    \ell\geq 1, \quad \mbox{and}\quad   \sum_{\ell=1}^\infty C_\ell<\infty.
\end{equation}
Then there holds
\begin{align*}
\sum_{j=1}^{n_x}\langle (P^{(j)})^*  \Lambda P^{(j)}a,a\rangle\leq
\sum_{\ell=1}^{\infty} C_\ell \langle \Lambda a,a\rangle,
\end{align*}
which yields Assumption \ref{ass:A}. 
The matrix $P$ quantifies the component-wise overlap between different singular vectors. In particular, $P_{t\ell}$ measures the interaction between $\varphi_t$ and $\varphi_\ell$ across the component spaces $\{X_j\}$. 
The value of $P_{t\ell}$ is small when the corresponding components $\varphi_{t,j}$ and $\varphi_{\ell,j}$ are nearly orthogonal in each $X_j$.
Both this stronger condition and Assumption \ref{ass:A} do not require different singular vectors to be component-wise orthogonal, but require the cross-scale interactions, as quantified by $P_{t\ell}$, to be sufficiently controlled. 
In other words, these assumptions prevents excessively strong and uncontrolled component-wise overlap between singular directions associated with very small singular values and those corresponding to much larger singular values.
\end{remark}

\begin{remark}\label{rem:tenosr_product_A}
An interesting special case of $A$ that satisfies Assumption \ref{ass:A} assumes a tensor product form and is given by 
\begin{align*}
Ax:= (V\otimes K) x=(\sum_{j=1}^{n_x} v_{1j} K x_j,\cdots,\sum_{j=1}^{n_x} v_{n_{y}j} K x_j )^\top, \quad \forall x \in X_1^{n_x},
\end{align*}
where $V=(v_{ij})$ is a $n_y\times n_x$ matrix with full column rank, $K: X_1 \to Y_1$ is a bounded linear operator, and $\otimes$ denotes the tensor product. 
Clearly, we can define $A_{i,j}x_j= v_{ij}Kx_j$.
Note that
\begin{align*}
A^* A = (V^* V) \otimes (K^* K)\quad \mbox{and} \quad D_{j}^{n_x} A^* A D_{j}^{n_x}= (D_{j}^{n_x}V^* V D_{j}^{n_x}) \otimes (K^* K).
\end{align*}
We denote the smallest eigenvalue of $V^*V$ by $\lambda_{\rm min}$, then there exists a constant $c_V>0$ such that $\lambda_{\rm min}\geq c_V$, since $V$ has full column rank. 
It follows that $$D_{j}^{n_x}V^* V D_{j}^{n_x}\preceq \|V\|^2 D_{j}^{n_x}\preceq \|V\|^2 I \preceq c_V^{-1}\|V\|^2 V^* V,$$ 
and hence 
\begin{align*}
D_{j}^{n_x} A^* A D_{j}^{n_x}= (D_{j}^{n_x}V^* V D_{j}^{n_x}) \otimes (K^* K)\preceq c_V^{-1}\|V\|^2(V^* V) \otimes (K^* K)=c_V^{-1}\|V\|^2A^* A.
\end{align*}
This specific class of linear forward operators is considered in \cite{RabanserNeumannHaltmeier:2019,JinLiu:2025} for the convergence analysis of (R)BCD. 
Such operators arise from various real-world applications, including multi-energy CT \cite{Zhang:2017,LiRavishankarLongFessler:2020}, spectral CT \cite{PerelliAndersen:2021}, electrocardiographic imaging \cite{CaraccioloCaiafa:2022} and dynamic magnetic resonance imaging \cite{BrahmaKolbitsch:2023}.
\end{remark}

The third main result presents the convergence rates for DSBGD in terms of the iteration number $k$ and the noise level $\delta$, under the additional source condition and the condition on $A$. 

\begin{theorem}[Convergence rates for noisy data]\label{thm:err-total}
Let Assumptions \ref{ass:stepsize}(iii), \ref{ass:source} and \ref{ass:A} be fulfilled with $c_0 \leq \alpha /2$, and $x_k^\delta$ be the DSBGD iterate defined in \eqref{eqn:DSBGD}. 
Then there holds
\begin{align*}
\E[\|x_k^\delta-x^\dag\|^2] \leq n c_* C_A k^{-\min(1,2\nu)(1-\alpha)}+ n c_{**} C_A k^{1-\alpha} \delta^2,
\end{align*}
where the constants $c_*$ and $c_{**}$ depends on $\nu$, $\|w\|$, $L$, $\alpha$ and $c_0$, but are independent of $n$, $k$ and $\delta$.
\end{theorem}

The next corollary follows directly from Theorem \ref{thm:err-total} and provides an order-optimal convergence rate of DSBGD for low-regularity solution $x^\dag$, under suitable \textit{a priori} stopping rules. 
The shorthand notation $k(\delta)=\mathcal{O}(\delta^{-s})$ means that there exist constants $0<C_1<C_2$ and some $\delta_0>0$, such that for any $\delta \in (0 ,\delta_0]$, there holds
$C_1\delta^{-s} \leq k(\delta) \leq C_2\delta ^{- s}.$
\begin{corollary}\label{cor:rate}
Let Assumptions \ref{ass:stepsize}(iii), \ref{ass:source} and \ref{ass:A} be fulfilled with $\nu\in(0,\frac12]$, then there holds
\begin{align*}
\E[\|x_{k(\delta)}^\delta-x^\dag\|^2]^\frac12\leq& \mathcal{O}(\delta^{\frac{2\nu}{1+2\nu}}),\quad \forall k(\delta)=\mathcal{O}(\delta^{-\frac{2}{(1+2\nu)(1-\alpha)}}).
\end{align*}  
\end{corollary}

\begin{remark}\label{rem:rate}
Note that the convergence rate $\mathcal{O}(\delta^{\frac{2\nu}{1+2\nu}}) $ stated in Corollary \ref{cor:rate} is identical with that for the Landweber method \cite{EnglHankeNeubauer:1996}, which is known to be order-optimal. 
Therefore, under suitable assumptions, DSBGD defined in \eqref{eqn:DSBGD} is also optimal with \textit{a priori} stopping rules when the solution regularity index $\nu\in(0,\frac12]$. 
\end{remark}

\section{Regularizing property}\label{sec:conv}    

In this section, we analyze the convergence of the DSBGD iterate defined in \eqref{eqn:DSBGD} for exact and noisy data in Theorems \ref{thm:conv-exact} and \ref{thm:conv-noisy} respectively. 
We first give several shorthand notation. We denote the DSBGD iterates for the exact data $y^\dag$ and noisy data $y^\delta$ by $x_k$ and $x_k^\delta$, respectively. 
For any $k=1,2,\cdots$, we define
\begin{align*}
&e_{k}=x_{k}-x^\dag,
\quad 
e_{k}^\delta=x_{k}^\delta-x^\dag,
\quad \Delta_{k}^\delta = x_{k}^\delta -x_{k},\\
&r_{k}=Ax_{k}-y^\delta,
\quad 
r_{k}^\delta=Ax_{k}^\delta-y^\delta, \quad d_{k}^\delta = r_{k}^\delta - r_{k},\\
&n=n_x n_y, \quad N_k= n^{-1}A^*- D_{j_k}^{n_x} A^* D_{i_k}^{n_y},\\
&B_X = n^{-1} A^* A: X\rightarrow X,\quad P^X_k = I_X-\eta_k B_X,\\
&B_Y = n^{-1} A A^*: Y\rightarrow Y, \quad P^Y_k = I_Y-\eta_k B_Y.
\end{align*}
We also define a constant \begin{equation}\label{eqn:tilde_L}
\tilde{L}=
\begin{cases}
    L ,& \mbox{under Assumption \ref{ass:stepsize}(i)},\\
    \min_{1\leq i\leq n_y}{\|A_i\|^2} ,&\mbox{under Assumption \ref{ass:stepsize}(ii)}.
\end{cases}
\end{equation} 
Without loss of generality, we assume $\tilde{L}>0$. 
Otherwise, if there exists a zero operator $A_t$, i.e., $\|A_t\|=0$, we can combine $A_t$ with another nonzero operator $A_{t'}$ and form a new nonzero linear operator $(A_t,A_{t'})^\top$. 
This yields a reformulation of the problem \eqref{eqn:lininv} with $n_y-1$ bounded linear operators $A_i$. 
For convenience in analysis, we reformulate the DSBGD iterate defined in \eqref{eqn:DSBGD} into an equivalent form:
\begin{align}\label{eqn:DSBGD2}
&x_{k+1}^\delta= x_k^\delta - \eta_k D_{j_k}^{n_x} A^* D_{i_k}^{n_y}(Ax_k^\delta - y^\delta), \\
\mbox{where}\quad &D_{j}^{n_x}:={\rm diag}(O_{X_1},\cdots,O_{X_{j-1}},I_{X_j},O_{X_{j-1}},\cdots,O_{X_{n_x}}),\nonumber\\
&D_{i}^{n_y}:={\rm diag}(O_{Y_1},\cdots,O_{Y_{i-1}},I_{Y_i},O_{Y_{i-1}},\cdots,O_{Y_{n_y}}),\nonumber
\end{align}
with $O_{Z}: Z \rightarrow Z$ and $I_{Z}: Z \rightarrow Z$ being the zero operator and the identity operator, respectively. 
Note that both $D_{j}^{n_x}: X \rightarrow {\bf 0}\times\cdots\times{\bf 0}\times X_j\times {\bf 0}\times\cdots\times{\bf 0}$ and $D_{i}^{n_y}: Y \rightarrow {\bf 0}\times\cdots\times{\bf 0}\times Y_i\times {\bf 0}\times\cdots\times{\bf 0}$ are orthogonal projectors, and there holds
\begin{equation}\label{eqn:D}
(D_{i}^{n_y})^*D_{i}^{n_y}=D_{i}^{n_y},\quad (D_{j}^{n_x})^*D_{j}^{n_x}=D_{j}^{n_x},\quad \E[D_{i}^{n_y}]=n_y^{-1}I_Y \quad \mbox{and} \quad \E[D_{j}^{n_x}]=n_x^{-1}I_X.
\end{equation}
Indeed, let $b_{j}=(O_{X_1},\cdots,O_{X_{j-1}},I_{X_j},O_{X_{j-1}},\cdots,O_{X_{n_x}})$, by the definitions of $D_{j}^{n_x}$, $D_{i}^{n_y}$, $A_i$, $A_{i,j}$ and
$x_{k+1}^\delta$ in \eqref{eqn:DSBGD},  we have
\begin{align*}
x_{k+1}^\delta=x_{k}^\delta - \eta_kA_{i_k,j_k}^*(A{x_k^\delta}-y^\delta)_{i_k}b_{j_k}= x_k^\delta - \eta_k D_{j_k}^{n_x} A^* D_{i_k}^{n_y}(Ax_k^\delta - y^\delta).
\end{align*}
Below we analyze the convergence of DSBGD iterate for exact and noisy data separately.

\subsection{Convergence for exact data}

We start by showing the convergence of the proposed DSBGD iterates to the exact solution $x^\dag$ for exact data $y^\dag$. 
To this end, we prove that both $\{\E[\|e_{k}\|^2]\}_{k\geq 1}$ and $\{x_k\}_{k\geq 1}$ are Cauchy sequences, and that the mean squared residual $\E[\|r_k\|^2]$ vanishes as $k\to\infty$ in the next two lemmas, i.e., Lemmas \ref{lem:mono} and \ref{lem:Cauchy}. 
The following lemma gives a crucial monotonicity result of the mean squared error $\E[\|e_{k}\|^2]$, which implies that $\{\E[\|e_{k}\|^2]\}_{k\geq 1}$ is a Cauchy sequence.

\begin{lemma}\label{lem:mono}
Let the stepsize $\eta_k =\mu_k L^{-1}$ for any $k\geq 1$ or $\eta_k =\mu_k \|A_{i_k}\|^{-2}$ for any $k\geq 1$, where $0<\mu_k\leq c_0<2$. 
Then for the DSBGD iterate $x_k$ defined in \eqref{eqn:DSBGD2} with exact data $y^\dag$,
the error $\{\E[\|e_{k}\|^2]\}_{k\geq 1}$ is monotonically decreasing and is a Cauchy sequence that converges to some constant $C_e\geq 0$ and 
$\lim_{k\to\infty}\sum_{t=k}^\infty{\mu_t}\E[\|r_t \|^2] = 0$.
\end{lemma}

\begin{proof}
By the definition of $x_{k+1}$ in \eqref{eqn:DSBGD2} and the identity $Ae_k=r_k$, 
the error $e_{k+1}$ is given by
\begin{equation}\label{eqn:e_k}
e_{k+1} = x_{k+1} -x^\dag= e_k - \eta_k D_{j_k}^{n_x} A^* D_{i_k}^{n_y} r_k.
\end{equation}
Then, with \eqref{eqn:D}, completing the square gives
\begin{align*}
    \| e_{k+1} \|^2 - \| e_k \|^2
    &= \eta_k^2\| D_{j_k}^{n_x} A^* D_{i_k}^{n_y}r_k\|^2 - 2\eta_k\langle e_k^\delta,D_{j_k}^{n_x} A^* D_{i_k}^{n_y} r_k\rangle\\
    &= \eta_k^2\langle D_{j_k}^{n_x} A^* D_{i_k}^{n_y}r_{k}^\delta, A^* D_{i_k}^{n_y}r_{k}^\delta\rangle - 2\eta_k\langle D_{i_k}^{n_y} A D_{j_k}^{n_x} e_{k}, D_{i_k}^{n_y} r_k\rangle.
\end{align*}
Taking conditional expectation $\E_X$ with respect to \(\mathcal{F}_k\), we obtain the following estimate
\begin{align*}
    \E_X[ \| e_{k+1} \|^2 - \| e_k \|^2 | \mathcal{F}_k ]
    &= n_x^{-1} \eta_k^2 \|A^* D_{i_k}^{n_y} r_k \|^2 - 2n_x^{-1} \eta_k\|D_{i_k}^{n_y} r_k \|^2\\
    &\leq n_x^{-1} \eta_k^2 \|A^* D_{i_k}^{n_y}\|^2\|D_{i_k}^{n_y} r_k \|^2 - 2n_x^{-1} \eta_k\|D_{i_k}^{n_y} r_k \|^2\\
    &\leq -n_x^{-1} \eta_k (2 - \mu_k)\|D_{i_k}^{n_y} r_k \|^2,
\end{align*}
where the last inequality follows from the inequalities $\eta_k\|A^* D_{i_k}^{n_y}\|^2=\eta_k\|A_{i_k}\|^2\leq \mu_k$ due to the conditions on $\eta_k$.
Together with the condition $\eta_k\geq L^{-1} \mu_k$, we obtain
\begin{align*}
    \E_X[ \| e_{k+1} \|^2 - \| e_k \|^2 | \mathcal{F}_k ]
    &\leq -n_x^{-1} L^{-1} \mu_k (2 - \mu_k)\|D_{i_k}^{n_y} r_k \|^2.
\end{align*}
Further, taking conditional expectation $\E$ with respect to \(\mathcal{F}_k\) gives
\begin{align*}
\E[ \| e_{k+1} \|^2 - \| e_k \|^2 | \mathcal{F}_k ]=\E\big[\E_X[ \| e_{k+1} \|^2 - \| e_k \|^2 | \mathcal{F}_k ]\,\big|\, \mathcal{F}_k \big]\leq -n^{-1} L^{-1} \mu_k (2 - \mu_k)\| r_k \|^2   
\end{align*}
Finally, taking full expectation on above estimates, with the condition $0<\mu_k \leq c_0 <2$, we obtain
\begin{align*}
\E[\| e_{k+1} \|^2] - \E[\| e_k \|^2]\leq -n^{-1} L^{-1} \mu_k (2 - c_0 )\E[\| r_k \|^2]\leq 0,  
\end{align*}
which implies that $\E[\|e_{k}\|^2]$ is monotonically decreasing. 
Thus, there holds for any $k\geq 1$ that
\begin{align*}
0\leq& \E[\| e_{k+1} \|^2]\leq \E[\| e_k \|^2]-n^{-1}L^{-1} \mu_k  (2 - c_0)\E[\| r_k \|^2]\\
\leq& \|e_1\|^2 - n^{-1}L^{-1} (2-c_0 )\sum_{t=1}^k \mu_t \E[\|r_t\|^2],
\end{align*}
which yields
\begin{align*}
\sum_{t=1}^\infty \mu_t\E[\|r_t\|^2]\leq (2-c_0)^{-1} nL\|e_1\|^2< \infty,
\end{align*}
and hence $\lim_{k\to\infty}\sum_{t=k}^\infty{\mu_t}\E[\|r_t\|^2] = 0$.
It follows that 
\begin{align*}
&\lim_{k<\ell,k\to\infty}\big|\E[\|e_{\ell}\|^2] -\E[\|e_{k}\|^2]\big|\leq\lim_{k<\ell,k\to\infty}\sum_{t=k}^{\ell-1}\big|\E[\|e_{t+1}\|^2] -\E[\|e_{t}\|^2]\big|\\
\leq& n^{-1} L^{-1}(2 - c_0)\lim_{k<\ell,k\to\infty}\sum_{t=k}^{\ell-1} \mu_t \E[\| r_t \|^2]
\leq n^{-1}L^{-1} (2 - c_0)\lim_{k\to\infty}\sum_{t=k}^{\infty} \mu_t \E[\| r_t \|^2]=0,
\end{align*}
which implies that $\{\E[\|e_{k}\|^2]\}_{k\geq 1}$ is a Cauchy sequence. Furthermore, the fact that $\E[\|e_{k}\|^2]\geq 0$ guarantees that $\lim_{k\to \infty}\E[\|e_{k}\|^2]:=C_e\geq 0$.
This completes the proof of the lemma.
\end{proof}

The next result shows that the sequence $\{x_k\}_{k\geq 1}$ is a Cauchy sequence and the mean squared residual $\lim_{k\to\infty}\E[\|r_k\|^2]=0$.

\begin{lemma}\label{lem:Cauchy}
Let the stepsize $\eta_k =\mu_k L^{-1}$ for any $k\geq 1$ or $\eta_k =\mu_k \|A_{i_k}\|^{-2}$ for any $k\geq 1$, where $0<\mu_k\leq c_0<2$. 
Then for exact data $y^\dag$,
the sequence $\{x_k\}_{k\geq 1}$ defined by \eqref{eqn:DSBGD2} is a Cauchy sequence. 
Further, let Assumption \ref{ass:stepsize}(i) or (ii) be fulfilled, and then $\lim_{k\to\infty}\E[\|r_k\|^2]=0$.
\end{lemma}

\begin{proof}
The argument below follows closely \cite[Lemma 3.3]{JinZhouZou:2020} and \cite[Proposition 3.2]{ZhouDSGD:2024}. For the convenience of readers, we state similar results here. 
For any $k' \ge k$, let $\ell$ such that $k' \geq \ell\geq k$ and $\E[\|r_\ell\|^2] = \min_{k\le t\le k'}\E[\|r_t\|^2]$.
We claim that 
\begin{equation}\label{eqn:cauchy}
\lim_{k'\geq k, k\to\infty}\E[\|x_{k'}-x_k\|^2]=\lim_{k'\geq k, k\to\infty}\E[\|e_{k'}-e_k\|^2]=0,    
\end{equation}
which implies that the sequence $\{x_k\}_{k\ge 1}$ is a Cauchy sequence. 
In fact, we can bound $\E[\|e_{k'}-e_k\|^2]$ with the Cauchy-Schwarz inequality
\begin{align*}
&\E[\|e_{k'} - e_k\|^2] \le 2\E[\|e_{k'} - e_\ell\|^2] + 2\E[\|e_{l} - e_k\|^2]\\
=& 4\E[\langle e_\ell-e_{k'},e_\ell\rangle]+4\E[\langle e_\ell-e_k,e_\ell\rangle] + 2(\E[\|e_{k'}\|^2] - \E[\|e_\ell\|^2]) + 2(\E[\|e_k\|^2] - \E[\|e_\ell\|^2]).
\end{align*}
Lemma \ref{lem:mono} implies that $$\lim_{k'\geq \ell, \ell\to\infty}(\E[\|e_{k'}\|^2] - \E[\|e_\ell\|^2])=0\; \mbox{  and } \lim_{\ell\geq k, k\to\infty}(\E[\|e_k\|^2] - \E[\|e_\ell\|^2])=0.$$
Hence, it remains to show that $\lim_{\ell\to \infty}\E[\langle e_\ell-e_{k'},e_\ell\rangle]=0$ and $\lim_{k\to \infty}\E[\langle e_\ell-e_k,e_\ell\rangle]=0$.
For the first estimate, with the identity \eqref{eqn:e_k}, we derive
\begin{equation*}
    e_\ell - e_{k'} = \sum_{t=\ell}^{k'-1}\eta_t D_{j_t}^{n_x} A^* D_{i_t}^{n_y} r_t.
\end{equation*}
Then by the triangle inequality and the inequality $\eta_t\leq \mu_t \tilde{L}^{-1}$, with $\tilde{L}$ defined in \eqref{eqn:tilde_L}, under Assumption \ref{ass:stepsize}(i) or (ii), there holds
\begin{align*}
    |\E[\langle e_\ell - e_{k'},e_\ell\rangle]|
    &\leq \tilde{L}^{-1} \sum_{t=\ell}^{k'-1}\mu_t | \E [\langle D_{j_t}^{n_x} A^* D_{i_t}^{n_y} r_t,e_\ell\rangle]|
    = n_x^{-1} \tilde{L}^{-1}\sum_{t=\ell}^{k'-1}\mu_t| \E [\langle A^* D_{i_t}^{n_y} r_t,e_\ell\rangle]|\\
    &= n_x^{-1}\tilde{L}^{-1}\sum_{t=\ell}^{k'-1}\mu_t| \E [\langle D_{i_t}^{n_y} r_t,r_\ell\rangle]|
    = n^{-1}\tilde{L}^{-1}\sum_{t=\ell}^{k'-1}\mu_t| \E [\langle r_t,r_\ell\rangle]|.
\end{align*}
This estimate, the Cauchy-Schwarz inequality, the choice of $\ell$ and Lemma \ref{lem:mono} further yield
\begin{align*}
    n \tilde{L}\lim_{\ell\to \infty}|\E[\langle e_\ell - e_{k'},e_\ell\rangle]|
    &\leq \lim_{\ell\to \infty}\sum_{t=\ell}^{k'-1}\mu_t\E[\|r_t\|^2]^{1/2}\E[\|r_\ell\|^2]^{1/2}
    \leq \lim_{\ell\to \infty}\sum_{t=\ell}^{k'-1}\mu_t\E[\|r_t\|^2]= 0.
\end{align*}
Similarly, we bound the second term $\E[\langle e_\ell-e_k,e_\ell\rangle]$ by
\begin{align*}
    |\E[\langle e_\ell-e_k,e_\ell\rangle]|
    &\leq \tilde{L}^{-1}\sum_{t=k}^{\ell-1}\mu_t| \E [\langle D_{j_t}^{n_x} A^* D_{i_t}^{n_y} r_t,e_\ell\rangle]|
    \leq n^{-1}\tilde{L}^{-1}\sum_{t=k}^{\ell-1}\mu_t\E[\|r_t\|^2].
\end{align*}
Finally, by Lemma \ref{lem:mono}, we have $\lim_{k\to \infty}\E[\langle e_\ell-e_k,e_\ell\rangle]=0$.
Thus, \eqref{eqn:cauchy} holds and the sequences $\{e_k\}_{k\geq 1}$ and $\{x_k\}_{k\geq1}$ are Cauchy sequences.

Now, by the triangle inequality and Cauchy-Schwarz inequality, we bound the difference between the residuals at $k$th and $k'$th iterates by 
\begin{align*}
\lim_{k,k'\to\infty}\big|\E[\|r_{k'}\|^2]^{\frac12} - \E[\|r_{k}\|^2]^{\frac12}\big|
&\le \lim_{k,k'\to\infty}\E[\|r_{k'} - r_{k}\|^2]^{\frac12}
= \lim_{k,k'\to\infty}\E[\|A(e_{k'} - e_{k})\|^2]^{\frac12}\\
&\le \lim_{k,k'\to\infty}\|A\|\E[\|e_{k'} - e_{k}\|^2]^{\frac12} =0.
\end{align*}
Thus, the sequences $\{r_{k}\}_{k\geq 1}$ and $\{\E[\|r_{k}\|^2]\}_{k\geq 1}$ are also Cauchy sequences.
Let $\lim_{k\to\infty}\E[\|r_k\|^2] = \epsilon\geq 0$. 
Assumption \ref{ass:stepsize}(i) or (ii) states $\sum_{t=1}^\infty\mu_t = \infty$, which implies
\begin{align*}
0 = \lim_{k\to\infty}\sum_{t=k}^\infty{\mu_t}\E[\|r_t \|^2] \ge \lim_{k\to\infty}\sum_{t=k}^\infty{\mu_t}\frac{1}{2}\epsilon,    
\end{align*}
and hence $\epsilon=0$.
This completes the proof of the lemma.
\end{proof}

Now, we can state the convergence of the DSBGD iterate defined in \eqref{eqn:DSBGD2} for exact data $y^\dag$.
\begin{theorem}[Convergence for exact data]\label{thm:conv-exact}
Let Assumption \ref{ass:stepsize}(i) or (ii) be fulfilled.
Then for exact data $y^\dag$, DSBGD sequence
$\{x_k\}_{k\geq1}$ defined in \eqref{eqn:DSBGD2} converges to $x^\dag$, i.e., 
\begin{equation*}
  \lim_{k\to\infty}\E[\|e_k\|^2]=\lim_{k\to\infty}\E[\|x_k-x^\dag\|^2] = 0.
\end{equation*}
\end{theorem}

\begin{proof}
By Lemma \ref{lem:Cauchy}, the sequence $\{r_{k}=Ax_k-y^\dag\}_{k\geq 1}$ is a Cauchy sequence that converges to $A x^*-y^\dag$ with $x^*:=\lim_{k\to \infty}x_k$ and $\E[\|A x^*-y^\dag\|]=\lim_{k\to\infty}\E[\|r_k\|^2]=0$, which imply that $x^*$ is a solution to problem \eqref{eqn:lininv}. The uniqueness of this solution completes the proof of the theorem.
\end{proof}

\subsection{Convergence for noisy data}

To establish the regularizing property of DSBGD, i.e., the convergence of DSBGD for noisy data $y^\delta$ at the stopping index $k_*=k(\delta)$, we bound the mean squared error $\E[\|e_{k_*}^\delta\|^2]=\E[\|x_{k_*}^\delta -x^\dag\|^2]$ by
\begin{align}\label{eqn:err_noise}
\E[\|e_{k_*}^\delta\|^2]\leq 2\E[\|e_{k_*}\|^2] + 2\E[\|\Delta_{k_*}^\delta\|^2],   
\end{align}
where the first term $\E[\|e_{k_*}\|^2]$ is estimated in Theorem \ref{thm:conv-exact}.
The following result concerns the second term $\E[\|\Delta_{k_*}^\delta\|^2]=\E[\|x_{k_*}^\delta-x_{k_*}\|^2]$ and gives the pathwise (i.e., along each realization in the filtration $\mathcal{F}_k$) stability of DSBGD iterate $x_k^\delta$ with respect to the noise level $\delta$ at $\delta=0$ for any $k\geq 1$.
\begin{prop}[Stability analysis]\label{prop:stability}
Let the stepsize $\eta_k =\mu_k L^{-1}$ for any $k\geq 1$ or $\eta_k =\mu_k \|A_{i_k}\|^{-2}$ for any $k\geq 1$, where $0<\mu_k\leq c_0<2$.
For any fixed $k\geq 1$ and any path
$(i_1,\ldots,i_{k-1})\in\mathcal{F}_k$, let $x_k$ and $x_k^\delta$ be DSBGD iterates along the path for exact data $y^\dag$ and noisy data $y^\delta$ respectively. Then there hold 
\begin{align*}
\E[\| \Delta_1^\delta \|^2]=0 \quad \mbox{and}\quad \E[\| \Delta_{k+1}^\delta \|^2] \leq n^{-1}\tilde{L}^{-1}\delta^2\sum_{t=1}^k \mu_t(2-\mu_t L)^{-1}.
\end{align*} 
\end{prop}

\begin{proof}
By the definitions $\Delta_k^\delta=x_k^\delta-x_k$ and $d_{k}^\delta = r_{k}^\delta - r_{k} = A \Delta_{k}^\delta - \xi$ and the recursion \eqref{eqn:DSBGD2}, for any fixed path $(i_1,\ldots,i_k)$, there holds
\begin{align*}
\Delta_{k+1}^\delta = \Delta_{k}^\delta - \eta_k D_{j_k}^{n_x} A^* D_{i_k}^{n_y} (A\Delta_{k}^\delta - \xi)=  \Delta_{k}^\delta - \eta_k D_{j_k}^{n_x} A^* D_{i_k}^{n_y} d_{k}^\delta.
\end{align*} 
Then, with \eqref{eqn:D}, completing the square yields 
\begin{align*}
    \| \Delta_{k+1}^\delta \|^2 - \| \Delta_k^\delta \|^2
    & = \eta_k^2\| D_{j_k}^{n_x} A^* D_{i_k}^{n_y}d_{k}^\delta\|^2 - 2\eta_k\langle\Delta_k^\delta,D_{j_k}^{n_x} A^* D_{i_k}^{n_y} d_{k}^\delta\rangle\\
    &= \eta_k^2\langle D_{j_k}^{n_x} A^* D_{i_k}^{n_y}d_{k}^\delta, A^* D_{i_k}^{n_y}d_{k}^\delta\rangle - 2\eta_k\langle D_{i_k}^{n_y} A D_{j_k}^{n_x}\Delta_{k}^\delta, D_{i_k}^{n_y} d_{k}^\delta\rangle.
\end{align*}
Taking conditional expectation $\E_X[\cdot]$ with respect to \(\mathcal{F}_k\), we obtain the following estimate:
\begin{align*}
    \E_X[ \| \Delta_{k+1}^\delta \|^2 - \| \Delta_k^\delta \|^2 | \mathcal{F}_k ]
    &\le n_x^{-1} \eta_k^2\|A^*D_{i_k}^{n_y}\|^2\| D_{i_k}^{n_y} d_k^\delta \|^2 - 2n_x^{-1}\eta_k\langle D_{i_k}^{n_y} A\Delta_{k}^\delta, D_{i_k}^{n_y} d_{k}^\delta\rangle\\
    &\le n_x^{-1} \eta_k \mu_k\| D_{i_k}^{n_y} d_k^\delta \|^2 - 2n_x^{-1}\eta_k\langle D_{i_k}^{n_y}d_{k}^\delta + D_{i_k}^{n_y}\xi, D_{i_k}^{n_y} d_{k}^\delta\rangle\\
    &\le n_x^{-1}\eta_k\big(-(2- \mu_k)\| D_{i_k}^{n_y} d_k^\delta \|^2+ 2\|D_{i_k}^{n_y} d_{k}^\delta\|\| D_{i_k}^{n_y}\xi\|\big)\\
    &\le n_x^{-1}\eta_k(2-\mu_k)^{-1}\|D_{i_k}^{n_y}\xi\|^2,
\end{align*}
where the last inequality is obtained by using Young's inequality $2ab\le \epsilon a^2 + \epsilon^{-1}b^2$ with $a = \| D_{i_k}^{n_y} d_{k}^\delta\|$, $b = \| D_{i_k}^{n_y}\xi\|$ and $\epsilon = 2-\mu_k$. 
Since $\eta_k\leq \mu_k \tilde{L}^{-1}$, we have
\begin{align*}
    \E_X[ \| \Delta_{k+1}^\delta \|^2 - \| \Delta_k^\delta \|^2 | \mathcal{F}_k ]
    &\le n_x^{-1}\tilde{L}^{-1} \mu_k(2-\mu_k)^{-1}\|D_{i_k}^{n_y}\xi\|^2.
\end{align*}
Moreover, we have
\begin{equation*}
\E[ \| \Delta_{k+1}^\delta \|^2 - \| \Delta_k^\delta \|^2 | \mathcal{F}_k]\leq n_x^{-1} \tilde{L}^{-1}\mu_k(2-\mu_k)^{-1}\E[\|D_{i_k}^{n_y}\xi\|^2 ]=n^{-1}\tilde{L}^{-1}\mu_k(2-\mu_k)^{-1}\delta^2.
\end{equation*}

Finally, by taking full expectation on the above estimate, we derive
\begin{align*}
\E[\| \Delta_{k+1}^\delta \|^2] &\leq \E[\| \Delta_k^\delta \|^2]+ n^{-1}\tilde{L}^{-1}\mu_k(2-\mu_k)^{-1}\delta^2\\
&\leq \E[\| \Delta_1^\delta \|^2]+ n^{-1}\tilde{L}^{-1}\delta^2\sum_{t=1}^k \mu_t(2-\mu_t)^{-1}
= n^{-1}\tilde{L}^{-1}\delta^2\sum_{t=1}^k \mu_t(2-\mu_t)^{-1},
\end{align*}
since $\E[\| \Delta_1^\delta \|^2]=\| x_1^\delta-x_1 \|^2=0$.
This completes the proof of the proposition.
\end{proof}

Now, we can give the proof of Theorem \ref{thm:conv-noisy} which gives the regularizing property of DSBGD under \textit{a priori} stopping rules.

\begin{proof}[Proof of Theorem \ref{thm:conv-noisy}]
Let 
$k_*=k(\delta)$ be the stopping index. 
Under Assumption \ref{ass:stepsize}(i) or (ii), the decomposition \eqref{eqn:err_noise} and the estimate in Proposition \ref{prop:stability} give that
\begin{align*}
\E[\|e_{k_*}^\delta\|^2]\leq 2\E[\|e_{k_*}\|^2] + 2n^{-1}\tilde{L}^{-1}(2-c_0)^{-1}\delta^2\sum_{t=1}^{{k_*-1}} \mu_t.
\end{align*}
Then, by Theorem \ref{thm:conv-exact}, when $\lim_{\delta\to0^+}k_*=\infty$ and $\lim_{\delta\to0^+}\delta^2\sum_{t=1}^{{k_*}} \mu_t = 0$, we have $\lim_{\delta\to0^+}\E[\|e_{k_*}^\delta\|^2]= 0$. This completes the proof.
\end{proof}

\subsection{Properties under the discrepancy principle}\label{sec:dp}

In this part, we discuss the behavior of DSBGD iterate when stopped by the discrepancy principle \eqref{eqn:discrepancy_1}, an \textit{a posteriori} stopping rule. 
Without loss of generality, we may assume $k(\delta)\geq 2$.
First, we present a result that guarantees the decay of the mean squared error under rule \eqref{eqn:discrepancy_1} before stopping in the following proposition. 

\begin{prop}[Monotonically decreasing property]\label{prop:decay_dp}
Let the stopping index $k_*=k(\delta)$ be chosen by \eqref{eqn:discrepancy_1} and the stepsize $\eta_k<2(1-\tau^{-1})L^{-1}$ for any $k\geq 1$. 
Then for any $1\leq k<k_*$,
the error $\E[\|e_{k}^\delta\|^2]$, associated with the DSBGD iterate $x_k^\delta$ defined in \eqref{eqn:DSBGD2} for noisy data $y^\delta$, is monotonically decreasing and $
\E[\| e_{k+1}^\delta \|^2 | \mathcal{F}_k ]\leq \| e_k^\delta \|^2 
$.   
\end{prop}
\begin{proof}
Similar to the analysis in the proof of Lemma \ref{lem:mono} and Proposition \ref{prop:stability}, 
the error $e_{k+1}^\delta$ is given by
\begin{equation}\label{eqn:e_k^delta}
e_{k+1}^\delta = e_k^\delta - \eta_k D_{j_k}^{n_x} A^* D_{i_k}^{n_y} r_k^\delta,
\end{equation}
and the difference between two successive iterate errors can be bounded by
\begin{align*}
    \E_X[\| e_{k+1}^\delta \|^2 - \| e_k^\delta \|^2 | \mathcal{F}_k ]
    &= \E_X[ \eta_k^2\| D_{j_k}^{n_x} A^* D_{i_k}^{n_y}r_k^\delta\|^2 - 2\eta_k\langle e_k^\delta,D_{j_k}^{n_x} A^* D_{i_k}^{n_y} r_k^\delta\rangle| \mathcal{F}_k ]\\
    &= n_x^{-1}\eta_k^2 \| A^* D_{i_k}^{n_y}r_{k}^\delta\|^2 - 2n_x^{-1}\eta_k\langle D_{i_k}^{n_y} A e_{k}^\delta, D_{i_k}^{n_y} r_k^\delta\rangle\\
    &\leq n_x^{-1}\eta_k^2 L\| D_{i_k}^{n_y}r_{k}^\delta\|^2 - 2n_x^{-1}\eta_k\langle D_{i_k}^{n_y} (r_{k}^\delta+\xi), D_{i_k}^{n_y} r_k^\delta\rangle\\
    &= n_x^{-1}\eta_k\big(-(2-\eta_k L)\| D_{i_k}^{n_y} r_{k}^\delta\|^2- 2\langle D_{i_k}^{n_y} \xi, D_{i_k}^{n_y} r_k^\delta\rangle\big),
\end{align*} 
and hence
\begin{align*}
\E[\| e_{k+1}^\delta \|^2 - \| e_k^\delta \|^2 | \mathcal{F}_k ]
&\leq n^{-1}\eta_k\big(-(2-\eta_k L)\| r_{k}^\delta\|^2- 2\langle \xi, r_k^\delta\rangle\big)\\
&\leq n^{-1}\eta_k\big(-(2-\eta_k L)\| r_{k}^\delta\|^2 +2 \|r_k^\delta\|\delta\big).
\end{align*}
For any $k<k_*$, there holds $\delta<\tau^{-1}\|r_k^\delta\|$, thus
\begin{align*}
\E[\| e_{k+1}^\delta \|^2 - \| e_k^\delta \|^2 | \mathcal{F}_k ]
&\leq -n^{-1}\eta_k(2-\eta_k L-2 \tau^{-1})\| r_{k}^\delta\|^2.
\end{align*}
Finally the condition $\eta_k<2(1-\tau^{-1})L^{-1}$ implies 
\begin{align*}
\E[\| e_{k+1}^\delta \|^2 | \mathcal{F}_k ] - \| e_k^\delta \|^2 
\leq 0 \quad \mbox{and}\quad \E[\| e_{k+1}^\delta \|^2] - \E[\| e_k^\delta \|^2] 
\leq 0,
\end{align*}
which complete the proof of the proposition.
\end{proof}

To establish the finite-iteration termination property of DSBGD equipped with the discrepancy principle \eqref{eqn:discrepancy_1}, we first derive the convergence rates of the mean squared residual $\E[\|r_{k}^\delta\|^2]=\|\E[r_{k+1}^\delta]\|^2+\E[\|r_{k}^\delta-\E[r_{k}^\delta]\|^2]$, its variance component $\E[\|r_{k}^\delta-\E[r_{k}^\delta]\|^2]$, and the mean residual $\|\E[r_{k+1}^\delta]\|$; see Appendix \ref{app:lem:r_k^delta} for the proof.

\begin{lemma}\label{lem:r_k^delta}
Let Assumption \ref{ass:stepsize}(iii) be fulfilled with $c_0 \leq \alpha /2$, then the residual
$r_k^\delta$ satisfies
\begin{align}
  \E[\|r_k^\delta\|^2] & \leq n c_1^* k^{-(1-\alpha)}+n c_2^* \delta^2,\label{eqn:r_k^delta}\\
  \|\E[r_k^\delta]\|&\leq \sqrt{\tfrac{nL}{2^\alpha e c_0}}\|e_1^\delta\|k^{-\frac{1-\alpha}{2}}+\delta, \label{eqn:E[r_k^delta]}\\
  \E[\|r_k^\delta-\E[r_k^\delta]\|^2]&\leq n c_1^* k^{-(1-\alpha)}+ n c_2^* k^{-\frac{\alpha}{2}}\delta^2, \quad \mbox{when} \quad \alpha\in(0,\frac23)\label{eqn:r_k^delta_var}
\end{align}
where the constants $c_1^*$ and $c_2^*$ depend on $L$, $\alpha$ and $c_0$, but are independent of $n$, $k$ and $\delta$.
\end{lemma}

The next result indicates the monotonicity of the mean residual $\|\E[r_{k}^\delta]\|$, which is also necessary for establishing the finite-iteration termination property in Theorem \ref{thm:termination}. 
Note that Assumption \ref{ass:stepsize}(i) is more general than that in (iii) used in Theorem \ref{thm:termination}.

\begin{lemma}\label{lem:mono_res}
Let Assumption \ref{ass:stepsize}(i) be fulfilled. Then for any $k\geq 1$, there holds
$$\|\E[r_{k+1}^\delta]\|\leq\|\E[r_{k}^\delta]\|.$$
\end{lemma}
\begin{proof}
By the fact $\E[N_k r_k^\delta|\mathcal{F}_k]=0$, taking full expectation on the recursion \eqref{eqn:recursion_r_0} gives
\begin{align*}
\E[r_{k+1}^\delta]
= P^Y_k \E[r_k^\delta].
\end{align*}
Then Assumption \ref{ass:stepsize}(i) implies $\|P^Y_k\|\leq 1$ and hence $\|\E[r_{k+1}^\delta]\|\leq \|\E[r_k^\delta]\|$.
This completes the proof of the lemma.
\end{proof}

Now, we can give the proof of Theorem \ref{thm:termination} which gives the finite-iteration termination property of DSBGD under rule \eqref{eqn:discrepancy_1}.
\begin{proof}[Proof of Theorem \ref{thm:termination}]
We decompose the residual $\|r^\delta_{k}\|$ by the triangle inequality as $$\|r^\delta_{k}\|\leq \|\E[r^\delta_{k}]\|+\|r^\delta_{k}-\E[r^\delta_{k}]\|,$$
which implies
\begin{align}\label{eqn:prob1}
\Prob\Big(\|r^\delta_{k}\|\leq \tau \delta\Big)
&\geq
\Prob\Big(\|\E[r^\delta_{k}]\|\leq \frac12(\tau+1) \delta\Big)\Prob\Big(\|r^\delta_{k}-\E[r^\delta_{k}]\|\leq \frac12(\tau-1) \delta\Big).
\end{align}
Following the analysis in the proof of \cite[Theorem 1.1]{JahnJin:2020}, we define 
\begin{align*}
k_0(\delta)=\min\{k\geq 1:\;\|\E[r_{k}^\delta]\|\leq \tfrac12(\tau+1) \delta\}.
\end{align*}
The estimate \eqref{eqn:E[r_k^delta]} gives
\begin{align*}
k_0(\delta)\leq\big[(c_{{\rm stop}}\delta^{-1})^{\frac{2}{1-\alpha}}\big]+1,
\end{align*}
with the constant $c_{{\rm stop}}=\frac{2\sqrt{nL}\|e_1^\delta\|}{\sqrt{2^\alpha e c_0}(\tau-1)}$ and $[\cdot]$ taking the integral part of a real number.
Then, we define \begin{equation}\label{eqn:def_bar_k}
\bar{k}(\delta):=\big[(c_{{\rm stop}}\delta^{-1})^{\frac{4}{2-3\alpha}}\big]+1
\geq k_0(\delta).    
\end{equation}
Thus, the monotonicity of the mean residual in Lemma \ref{lem:mono_res} gives 
\begin{align}\label{eqn:mean_r_bar_k}
\|\E[r_{\bar{k}(\delta)}^\delta]\|\leq \frac12 (\tau+1)\delta.   
\end{align}
By the definition of $k(\delta)$ in the discrepancy principle \eqref{eqn:discrepancy_1}, the event $\|r^\delta_{\bar{k}(\delta)}\|\leq \tau \delta$ implies 
\begin{align*} 
k(\delta)=\min\{k\geq 1:\;\|r^\delta_{k}\|\leq \tau \delta\}\leq \bar{k}(\delta),
\end{align*}
and hence
\begin{align*}
\Prob\Big(k(\delta)\leq \bar{k}(\delta)\Big)\geq \Prob\Big(\|r^\delta_{\bar{k}(\delta)}\|\leq \tau \delta\Big).
\end{align*}
Further, the estimates \eqref{eqn:prob1} and \eqref{eqn:mean_r_bar_k} yield
\begin{align*}
\Prob\Big(k(\delta)\leq \bar{k}(\delta)\Big)
\geq& \Prob\Big(\|r^\delta_{\bar{k}(\delta)}-\E[r^\delta_{\bar{k}(\delta)}]\|\leq \tfrac12(\tau-1) \delta\Big)
= 1-\Prob\Big(\|r^\delta_{\bar{k}(\delta)}-\E[r^\delta_{\bar{k}(\delta)}]\|> \tfrac12(\tau-1) \delta\Big).
\end{align*}
By Chebyshev's inequality, the estimate \eqref{eqn:r_k^delta_var} and the definition \eqref{eqn:def_bar_k}, we have 
\begin{align*}
&\Prob\Big(\|r^\delta_{\bar{k}(\delta)}-\E[r^\delta_{\bar{k}(\delta)}]\|> \frac12(\tau-1) \delta\Big)
\leq 4(\tau-1)^{-2}\delta^{-2}\E[\|r^\delta_{\bar{k}(\delta)}-\E[r^\delta_{\bar{k}(\delta)}]\|^2]\\
\leq& 4(\tau-1)^{-2}n\big(c_1^* \bar{k}(\delta)^{-(1-\alpha)}\delta^{-2}+ c_2^* \bar{k}(\delta)^{-\frac{\alpha}{2}}\big)\\
\leq& 4(\tau-1)^{-2}n\big(c_1^* (c_{{\rm stop}}\delta^{-1})^{-\frac{4(1-\alpha)}{2-3\alpha}}\delta^{-2}+ c_2^* (c_{{\rm stop}}\delta^{-1})^{-\frac{2\alpha}{2-3\alpha}}\big)\\
\leq& 4(\tau-1)^{-2}n\big(c_1^* c_{{\rm stop}}^{-\frac{4(1-\alpha)}{2-3\alpha}}+ c_2^* c_{{\rm stop}}^{-\frac{2\alpha}{2-3\alpha}}\big)\delta^{\frac{2\alpha}{2-3\alpha}},
\end{align*}
and hence
\begin{align*}
\Prob\Big(k(\delta)\leq \bar{k}(\delta)\Big)
\geq 1- 4(\tau-1)^{-2}n(c_1^*+c_2^*)\big(c_{{\rm stop}}^{-\frac{4(1-\alpha)}{2-3\alpha}}+  c_{{\rm stop}}^{-\frac{2\alpha}{2-3\alpha}}\big)\delta^{\frac{2\alpha}{2-3\alpha}}.
\end{align*}
This gives the first estimate in this theorem.

Similarly, we define 
\begin{equation}\label{eqn:def_tilde_k}
\tilde{k}(\delta):=\big[(c_{{\rm stop}}\delta^{-1})^{\frac{2}{1-\alpha}}\ln(e+\delta^{-1})\big]+1
\geq k_0(\delta),    
\end{equation}
and then
\begin{align*}
\Prob\Big(k(\delta)\leq \tilde{k}(\delta)\Big)
\geq& 1-\Prob\Big(\|r^\delta_{\tilde{k}(\delta)}-\E[r^\delta_{\tilde{k}(\delta)}]\|> \tfrac12(\tau-1) \delta\Big),
\end{align*}
where
\begin{align*}
&\Prob\Big(\|r^\delta_{\tilde{k}(\delta)}-\E[r^\delta_{\tilde{k}(\delta)}]\|> \frac12(\tau-1) \delta\Big)
\leq 4(\tau-1)^{-2}\delta^{-2}\E[\|r^\delta_{\tilde{k}(\delta)}-\E[r^\delta_{\tilde{k}(\delta)}]\|^2]\\
\leq& 4(\tau-1)^{-2}n\big(c_1^* (c_{{\rm stop}}\delta^{-1})^{-2}\ln^{-(1-\alpha)}(e+\delta^{-1})\delta^{-2}+ c_2^* (c_{{\rm stop}}\delta^{-1})^{-\frac{\alpha}{1-\alpha}}\ln^{-\frac{\alpha}{2}}(e+\delta^{-1})\big)\xrightarrow{\;\;\;\delta \to 0^+\;\;} 0,
\end{align*}
which implies the finite-iteration termination property:
\begin{align*}
\lim_{\delta\to 0^+}\mathbb{P}\Big(k(\delta)\leq \tilde{k}(\delta)\Big)=1.
\end{align*}
This completes the proof of the theorem.
\end{proof}

\section{Convergence rates}\label{sec:rate}
In this section, we derive convergence rates for DSBGD under Assumptions \ref{ass:stepsize}(iii),  \ref{ass:source}, and
\ref{ass:A}. 
First, we bound the error $\E[\|e_k^\delta\|^2]$ in terms of the residual $\E[\|r_k^\delta\|^2]$.
\begin{lemma}
Let Assumptions \ref{ass:source} and \ref{ass:A} be fulfilled, then there holds
\begin{align}\label{eqn:err_res}
\E[\|e_{k+1}^\delta\|^2]
&\leq 2 \|w\|^2 \phi_{0,\nu}^2 +\frac{2}{n} \bigg(\sum_{\ell=1}^k \eta_\ell \phi_{\ell,\frac12}\bigg)^2 \delta^2
+C_A\sum_{\ell=1}^k \eta_\ell^2  \phi_{\ell,\frac12}^2 \E[\|r_\ell^\delta \|^2],
\end{align}
where $\phi_{\ell,s}=\big\| \prod_{t=\ell+1}^k P^X_t  B_X^s\big\|\leq (\frac{s}{e})^s(\sum_{t=\ell+1}^k\eta_t)^{-s}$ with the convention $\phi_{k,s}=\|B_X^s\|$.
\end{lemma}

\begin{proof}
With the definitions $P^Y_k$, $N_k$ and $r_k^\delta=Ae_k^\delta -\xi$, it follows from the recursion \eqref{eqn:e_k^delta} that
\begin{align*}
e_{k+1}^\delta &= e_k^\delta - \eta_k D_{j_k}^{n_x} A^* D_{i_k}^{n_y} (Ae_k^\delta -\xi)
= P^X_k e_k^\delta +\eta_k N_k Ae_k^\delta+\eta_k D_{j_k}^{n_x} A^* D_{i_k}^{n_y} \xi\\
&= P^X_k e_k^\delta +\eta_k N_k Ae_k^\delta - \eta_k N_k\xi +  n^{-1}\eta_k A^* \xi
= P^X_k e_k^\delta +\eta_k N_k r_k^\delta +  n^{-1}\eta_k A^* \xi.
\end{align*}
Similar to the analysis in the proof of Lemma \ref{lem:r_k^delta}, 
we obtain
\begin{align*}
e_{k+1}^\delta= \prod_{t=1}^k P^X_t e_{1}^\delta +\sum_{\ell=1}^k \eta_\ell \prod_{t=\ell+1}^k P^X_t  N_\ell r_\ell^\delta + \frac1n\sum_{\ell=1}^k \eta_\ell \prod_{t=\ell+1}^k P^X_t  A^*\xi.
\end{align*}
Then by the bias-variance decomposition, Cauchy-Schwarz inequality and the identity $$\E[\langle N_{t}r_t^\delta, N_{\ell}r_\ell^\delta\rangle|\mathcal{F}_t]=0$$ for any $t>\ell$, we have
\begin{align*}
\E[\|e_{k+1}^\delta\|^2]&= \bigg\|\prod_{t=1}^k P^X_t e_{1}^\delta +\frac1n\sum_{\ell=1}^k \eta_\ell \prod_{t=\ell+1}^k P^X_t  A^*\xi\bigg\|^2+\E\bigg[\bigg\|\sum_{\ell=1}^k \eta_\ell \prod_{t=\ell+1}^k P^X_t  N_\ell r_\ell^\delta \bigg\|^2\bigg]\\
&\leq 2\bigg\|\prod_{t=1}^k P^X_t e_{1}^\delta\bigg\|^2 +\frac{2}{n^2} \bigg\|\sum_{\ell=1}^k \eta_\ell \prod_{t=\ell+1}^k P^X_t  A^*\bigg\|^2 \delta^2
+\sum_{\ell=1}^k \eta_\ell^2 \E\bigg[\bigg\|  \prod_{t=\ell+1}^k P^X_t  N_\ell r_\ell^\delta \bigg\|^2\bigg].
\end{align*}
Under Assumption \ref{ass:source}, with the definition $B_X = n^{-1} A^* A$, the triangle inequality and Lemma \ref{lem:N} (under Assumption \ref{ass:A}), there holds
\begin{align*}
\E[\|e_{k+1}^\delta\|^2]
&\leq 2 \|w\|^2 \phi_{0,\nu}^2 +\frac{2}{n} \bigg(\sum_{\ell=1}^k \eta_\ell \phi_{\ell,\frac12}\bigg)^2 \delta^2
+C_A\sum_{\ell=1}^k \eta_\ell^2  \phi_{\ell,\frac12}^2 \E[\|r_\ell^\delta \|^2].
\end{align*}
Finally Lemma \ref{lem:estimate-B} completes the proof.
\end{proof}

Now, we give the proof of Theorem \ref{thm:termination} which provides the convergence rates of DSBGD.

\begin{proof}[Proof of Theorem \ref{thm:err-total}]
We estimate each components of the upper bound in \eqref{eqn:err_res}.
For the first term, Lemma \ref{lem:estimate-B} and the inequality $k+1\leq 2k$ yield
\begin{align*}
\phi_{0,\nu}\leq \Big(\frac{\nu L}{e c_0}\Big)^\nu k^{-\nu(1-\alpha)}\leq \Big(\frac{2^{1-\alpha}\nu L}{e c_0}\Big)^\nu (k+1)^{-\nu(1-\alpha)}\leq \Big(\frac{\nu L}{c_0}\Big)^\nu (k+1)^{-\nu(1-\alpha)}.
\end{align*}
For the second term, the inequality \eqref{eqn:bdd-phi-r} in Lemma \ref{lem:sum} and $\|A\|^2\leq nL$ imply
\begin{align*}
\sum_{\ell=1}^k \eta_\ell \phi_{\ell,\frac12}
&\leq \eta_k \|B_X^s\|+(2e)^{-\frac12}\sum_{\ell=1}^{k-1} \frac{\eta_\ell}{(\sum_{t=\ell+1}^k\eta_t)^{\frac12}}
\leq c_0 L^{-1} \|B_X^s\| k^{-\alpha} +(2e)^{-\frac12}
c_0^{\frac12}L^{-\frac12} B(\tfrac12,1-\alpha) k^{\frac{1-\alpha}{2}}\\
&\leq \big(n^{-1}c_0L^{-1} \|A\|^2 +(2e)^{-\frac12}
c_0^{\frac12}L^{-\frac12} B(\tfrac12,1-\alpha) \big)(k+1)^{\frac{1-\alpha}{2}}\\
&\leq \big(c_0  +(2e)^{-\frac12}
c_0^{\frac12}L^{-\frac12} B(\tfrac12,1-\alpha) \big)(k+1)^{\frac{1-\alpha}{2}}.
\end{align*}
Next, we bound the last term. By the definitions of $P^X_t$, $P^Y_t$, $B_X$, and $B_Y$, there holds
\begin{align*}
\bigg\| \prod_{t=\ell+1}^k P^X_t  B_X^\frac12\bigg\|=\bigg\| \prod_{t=\ell+1}^k P^Y_t  B_Y^\frac12\bigg\|,
\end{align*}
and thus
\begin{align*}
\sum_{\ell=1}^k \eta_\ell^2  \phi_{\ell,\frac12}^2 \E[\|r_\ell^\delta \|^2]=\sum_{\ell=1}^k \eta_\ell^2  \bigg\| \prod_{t=\ell+1}^k P^X_t  B_X^\frac12\bigg\|^2 \E[\|r_\ell^\delta \|^2]
=L^{-1}\mathfrak{I}'_2,
\end{align*}
with $\mathfrak{I}'_2$ given in \eqref{eqn:I_2_for_rate}. 
By the upper bound \eqref{eqn:I_2} of $\mathfrak{I}'_2$ derived in the proof of Lemma \ref{lem:r_k^delta} in Appendix \ref{app:lem:r_k^delta} and the inequalities \eqref{eqn:I_22}, \eqref{eqn:bdd-log}, \eqref{eqn:I_21_2} and $k+1\leq 2k$, we have
\begin{align*}
L^{-1}\mathfrak{I}'_2 \leq c_0^2 (n_x L)^{-1}   (nc_1^* k^{-(1+\alpha)}+nc_2^*k^{-2\alpha}\delta^2)
+(2e)^{-1}(nc_1^* \mathfrak{I}_{2,1}+nc_2^*\delta^2 \mathfrak{I}_{2,2}),
\end{align*}
where 
\begin{align*}
\mathfrak{I}_{2,1}&
\leq
2^3 \alpha^{-1}c_0L^{-1}k^{-(1-\alpha)}\leq 2^4 c_0 L^{-1} \alpha^{-1}(k+1)^{-(1-\alpha)},\\
\mathfrak{I}_{2,2}&
\leq 2^{2}c_0L^{-1} k^{1-\alpha}+ c_0L^{-1} 
\begin{cases}
2 k^{1-\alpha}, & \mbox{when}\; \alpha = \frac12,\\
\tfrac{2^{2}}{|1-2\alpha|}, & \mbox{otherwise}.
\end{cases}
\end{align*}
Finally, by combining the above estimates, we derive from \eqref{eqn:err_res} that
\begin{align*}
\E[\|e_{k+1}^\delta\|^2]
&\leq 2 \|w\|^2 \phi_{0,\nu}^2 +\frac{2}{n} \bigg(\sum_{\ell=1}^k \eta_\ell \phi_{\ell,\frac12}\bigg)^2 \delta^2
+C_A\sum_{\ell=1}^k \eta_\ell^2  \phi_{\ell,\frac12}^2 \E[\|r_\ell^\delta \|^2]\\
&\leq 2 \|w\|^2 \nu^{2\nu}c_0^{-2\nu}L^{2\nu}(k+1)^{-2\nu(1-\alpha)}
+2 n^{-1} \big(c_0 +(2e)^{-\frac12}
c_0^{\frac12}L^{-\frac12} B(\tfrac12,1-\alpha) \big)^2 (k+1)^{1-\alpha} \delta^2\\
&\quad +C_A\Big(c_0^2 (n_x L)^{-1}   (nc_1^* k^{-(1+\alpha)}+nc_2^*k^{-2\alpha}\delta^2)
+(2e)^{-1}(nc_1^* \mathfrak{I}_{2,1}+nc_2^*\delta^2 \mathfrak{I}_{2,2})\Big)\\
&\leq n c_* C_A (k+1)^{-\min(1,2\nu)(1-\alpha)}+ n c_{**} C_A (k+1)^{1-\alpha} \delta^2
\end{align*}
with some the constants $c_*$ and $c_{**}$ depending on $\nu$, $\|w\|$, $L$, $\alpha$ and $c_0$, but are independent of $n$, $k$ and $\delta$.
This completes the proof of the theorem.
\end{proof}

\section{Numerical experiments}\label{sec:num}

In this section, we provide numerical experiments for several linear inverse problems in Hilbert spaces to complement the analysis.

~\\
\noindent\emph{\textbf{Example 5.1.}}
First, we employ three model examples, i.e., \texttt{phillips} (mildly ill-posed), \texttt{gravity} (severely ill-posed) and \texttt{shaw}
(severely ill-posed), which are all Fredholm integral equations of the first kind over an interval $I$: 
\begin{equation}\label{eqn:Fredholm}
    Ax(s) = \int_I k(s,t)x(t){\rm d}t = y(s), \quad s\in I,
\end{equation}
with the kernel $k(s,t): I\times I\to\mathbb{R}^+$. The examples are generated from the code \texttt{phillips}, \texttt{gravity} and \texttt{shaw}, taken from the \texttt{MATLAB} package Regutools \cite{P.C.Hansen2007} (publicly available at \url{http://people.compute.dtu.dk/pcha/Regutools/}, last accessed on August 20, 2020). 
All the examples are discretized into a linear problem of the form \eqref{eqn:lininv} of the size $n_y = n_x = 100$ with $A_{i,j}: \mathbb{R}^{10} \rightarrow \mathbb{R}^{10}$ for any $1\leq i\leq n_y$ and $1\leq j\leq n_x$. 

We generate the exact solution $x^\dag$ by $x^\dag = \|x_e\|_{\ell^\infty}^{-1}x_e$,
with $x_e$ being the exact solution provided by the package and $\|\cdot\|_{\ell^\infty}$ the maximum norm of a vector. 
The exact data $y^\dag$ is given by $y^\dag=A x^\dag$ and the noisy data $y^\delta$ is generated by
\begin{equation}\label{eqn:num_y_data}
y^\delta_i:=y^\dag_i+\epsilon\|y^\dag\|_{\ell^\infty}\xi_i,\quad i=1,\cdots,n,
\end{equation}
where the random variables $\xi_i$ follow the standard Gaussian distribution, and $\epsilon > 0$ is the relative noise level. 

For comparison, we adopt the RBCD method proposed in \cite{JinLiu:2025} and the benchmark Landweber method (LM), both employing a constant stepsize with $\eta_k = \|A\|_F^{-2}$.
Among constant stepsize schedules for RBCD, this choice allows the method to achieve its best accuracy without requiring many more iterations than LM. 
Our proposed method, DSBGD, employs two stepsize rules: a constant stepsize with $\eta_k = c_0\max_{1\leq i\leq n_y}\|A_i\|^{-2}$, referred to as DSBGD-S1; and an adaptive stepsize with $\eta_k = c_0\|A_{i_k}\|^{-2}$, where $i_k$ is chosen uniformly at random from $\{1,\ldots,n_y$\} as in the DSBGD iteration \eqref{eqn:DSBGD}, referred to as DSBGD-S2. 
For \texttt{phillips} and \texttt{gravity}, we set $c_0 = 0.05$ in both DSBGD-S1 and DSBGD-S2; for \texttt{shaw}, we set $c_0 = 0.1$ in DSBGD-S1 and $c_0 \in \{0.1,0.002\}$ in DSBGD-S2.
The stepsizes for DSBGD method are chosen to balance optimal convergence and computational efficiency across all noise levels. 
Note the benchmark method, LM, can achieve order optimality \cite{HankeNeubauerScherzer:1995}. All iterative methods are initialized to zero. The statistical quantities generated by DSBGD and RBCD are computed based on 50 independent runs.

The numerical comparison results for the three examples, with varying noise levels $\epsilon$, are presented in Tables \ref{tab:phil}-\ref{tab:shaw}. The stopping index $k_*=k(\delta)$, measured in epochs, is chosen by the discrepancy principle \eqref{eqn:discrepancy_1} with $\tau=1.1$ for all methods. Here, one epoch refers to one Landweber iteration, $n_x n_y$ DSBGD iterations and $n_x$ RBCD iterations, respectively, so that their overall computational complexity (excluding the cost of checking the stopping criterion) is comparable. 
The quantity $e_*$ denotes the square root of the relative mean squared error, defined by $e_* = \mathbb{E}[ \| x_{k_*}^{\delta} - x^{\dagger} \|^2 ]^{1/2} / \| x^{\dagger} \|$ for DSBGD and RBCD, and as $\| x_{k_*}^{\delta} - x^{\dagger} \|/\| x^{\dagger} \|$ for LM. As the noise level $\epsilon$ increases, all methods terminate earlier and the reconstruction error increases. LM achieves the smallest error, but requires a moderate number of iterations. For all these three examples, both DSBGD-S1 and DSBGD-S2 clearly outperform RBCD in terms of accuracy, while requiring far fewer iterations than both LM and RBCD.
Compared with RBCD, the DSBGD variants thus offer a favorable trade-off between reconstruction accuracy and computational efficiency. 
For larger relative noise level $\epsilon$, the reconstruction accuracy of DSBGD under both stepsize schedules is comparable to that of LM. 
However, as $\epsilon$ decreases, the accuracy of DSBGD does not decay at a rate comparable to LM, suggesting that the current stepsizes may not allow DSBGD to achieve an optimal convergence rate (in terms of the noise level) as LM.

\begin{table}[htp!]
  \centering
  \setlength{\tabcolsep}{3pt}
  \begin{threeparttable}
  \caption{The comparison under discrepancy principle for \texttt{phillips}.}\label{tab:phil}
    \begin{tabular}{cccccccccccc}
    \toprule
    \multicolumn{1}{c}{Method}&
    \multicolumn{2}{c}{DSBGD-S1}&
    \multicolumn{2}{c}{DSBGD-S2}&
    \multicolumn{2}{c}{RBCD}&
    \multicolumn{2}{c}{LM}\\
    \cmidrule(lr){2-3} \cmidrule(lr){4-5} \cmidrule(lr){6-7} \cmidrule(lr){8-9}
    
    $\epsilon$ & $k_*$ & $e_*$ & $k_*$ & $e_*$ &$ k_*$ & $e_*$ & $k_*$ & $e_*$ \\
    \midrule
        1e-3   &  180.4  &   1.53e-1  & 149.6  &  1.56e-1   & 5260.3  & 3.38e-1  & 421.0  & 2.43e-2 \cr
        5e-3   &  37.6	 &   1.59e-1  & 35.9   &  1.57e-1   & 190.9   & 3.46e-1  & 178.0  & 4.00e-2 \cr
        1e-2   &  17.0   &   1.66e-1  & 18.1   &  1.67e-1   & 92.1    & 3.58e-1  & 79.0   & 6.77e-2 \cr
        5e-2   &  4.5    &   1.98e-1  & 4.3    &  2.01e-1   & 21.9    & 3.68e-1  & 22.0   & 1.25e-1 \cr
    \bottomrule
    \end{tabular}
    \end{threeparttable}
\end{table}

\begin{table}[hbt!]
  \centering
  \setlength{\tabcolsep}{3pt}
  \begin{threeparttable}
  \caption{The comparison under discrepancy principle for \texttt{gravity}.}\label{tab:grav}
    \begin{tabular}{cccccccccccc}
    \toprule
    \multicolumn{1}{c}{Method}&
    \multicolumn{2}{c}{DSBGD-S1}&
    \multicolumn{2}{c}{DSBGD-S2}&
    \multicolumn{2}{c}{RBCD}&
    \multicolumn{2}{c}{LM}\\
    \cmidrule(lr){2-3} \cmidrule(lr){4-5} \cmidrule(lr){6-7} \cmidrule(lr){8-9}
    
    $\epsilon$ & $k_*$ & $e_*$ & $k_*$ & $e_*$ &$ k_*$ & $e_*$ & $k_*$ & $e_*$ \\
    \midrule
        1e-3   &  240.0	 &   1.53e-1  & 180.0   &  1.53e-1   &  3280.3 &  4.98e-1  & 928.0  & 3.08e-2 \cr
        5e-3   &  32.0   &   1.59e-1  & 26.0  	&  1.60e-1   &  230.9  &  4.89e-1  & 140.0  & 5.56e-2 \cr
        1e-2   &  15.5   &   1.71e-1  & 13.3	&  1.65e-1   &   84.6  &  4.88e-1  & 69.0   & 7.15e-2 \cr
        5e-2   &  3.0    &   2.06e-1  & 3.1	    &  2.09e-1   &   15.7  &  5.17e-1  & 14.0  & 1.42e-1 \cr
    \bottomrule
    \end{tabular}
    \end{threeparttable}
\end{table}

DSBGD-S1 and DSBGD-S2 perform comparably for \texttt{phillips} and \texttt{gravity}. For \texttt{shaw}, however, the two variants differ noticeably. Note that DSBGD-S2 with $c_0=0.02$ achieves a reconstruction accuracy comparable to or even better than that of LM, whereas with $c_0=0.1$ it stops much earlier but yields a less accurate reconstruction. 
Therefore, under suitable adaptive stepsizes, DSBGD-S2 can attain an accuracy comparable to that of DSBGD-S1 with far fewer iterations. 
This observation can be attributed to the distribution of the row-block norms shown in Figure \ref{fig:A_ik}. For \texttt{phillips} and \texttt{gravity} the norms of blocks
are concentrated near their maximum value employed for DSBGD-S1, which explains why DSBGD-S1 and DSBGD-S2 behave similarly. For \texttt{shaw}, in contrast, the norms vary strongly, making the adaptive stepsize in DSBGD-S2 much larger than that in DSBGD-S1 with the same $c_0$ for most iterations. This observation suggests that the adaptive stepsize can better exploit the local geometry of the problem as iterations proceed, and hence achieve better accuracy than DSBGD-S1 with a smaller $c_0$ without requiring more iterations.

\begin{figure}[hbt!]
    \centering
    \setlength{\tabcolsep}{0pt}
    \begin{tabular}{ccc}
    \includegraphics[height=.25\textwidth]{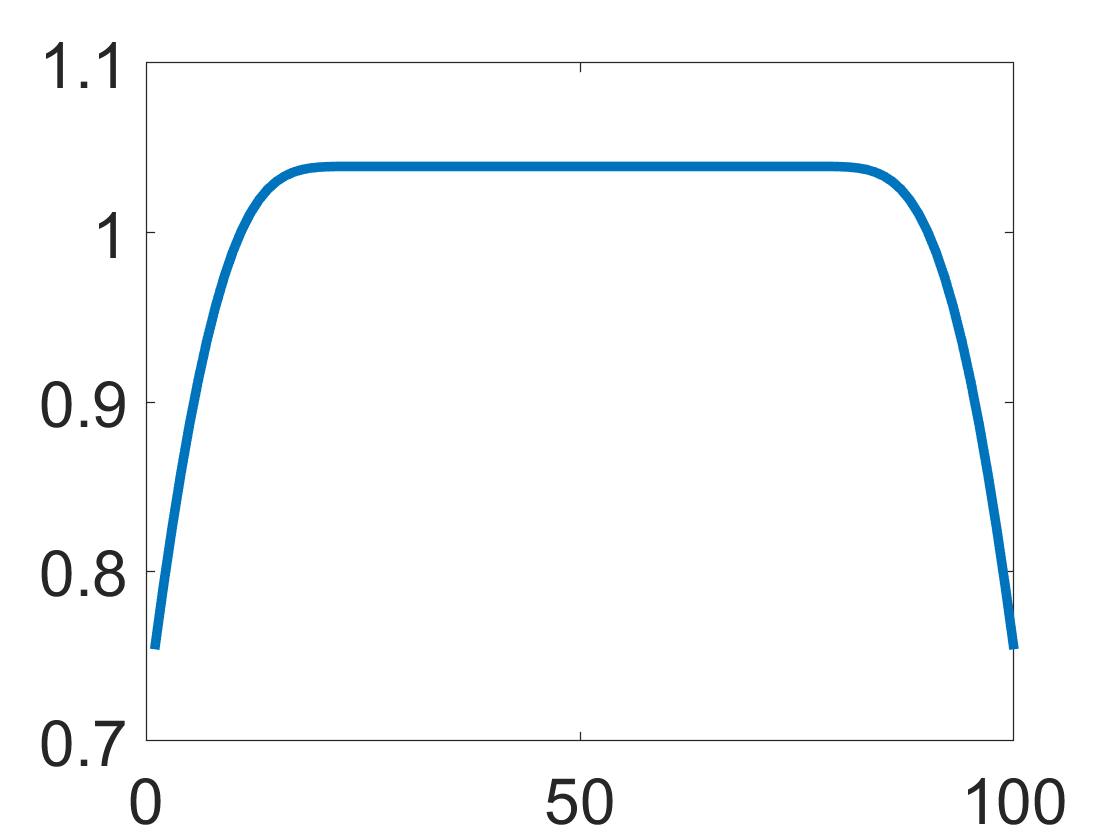}&
    \includegraphics[height=.25\textwidth]{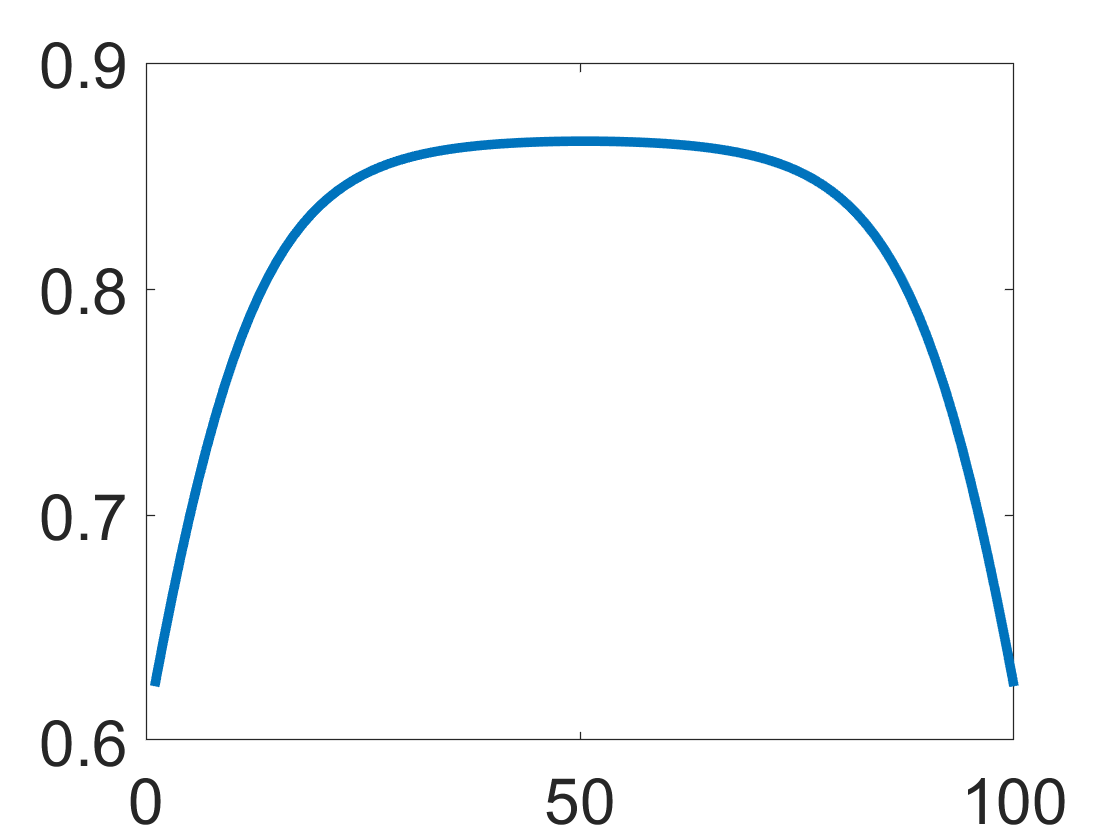}&
    \includegraphics[height=.25\textwidth]{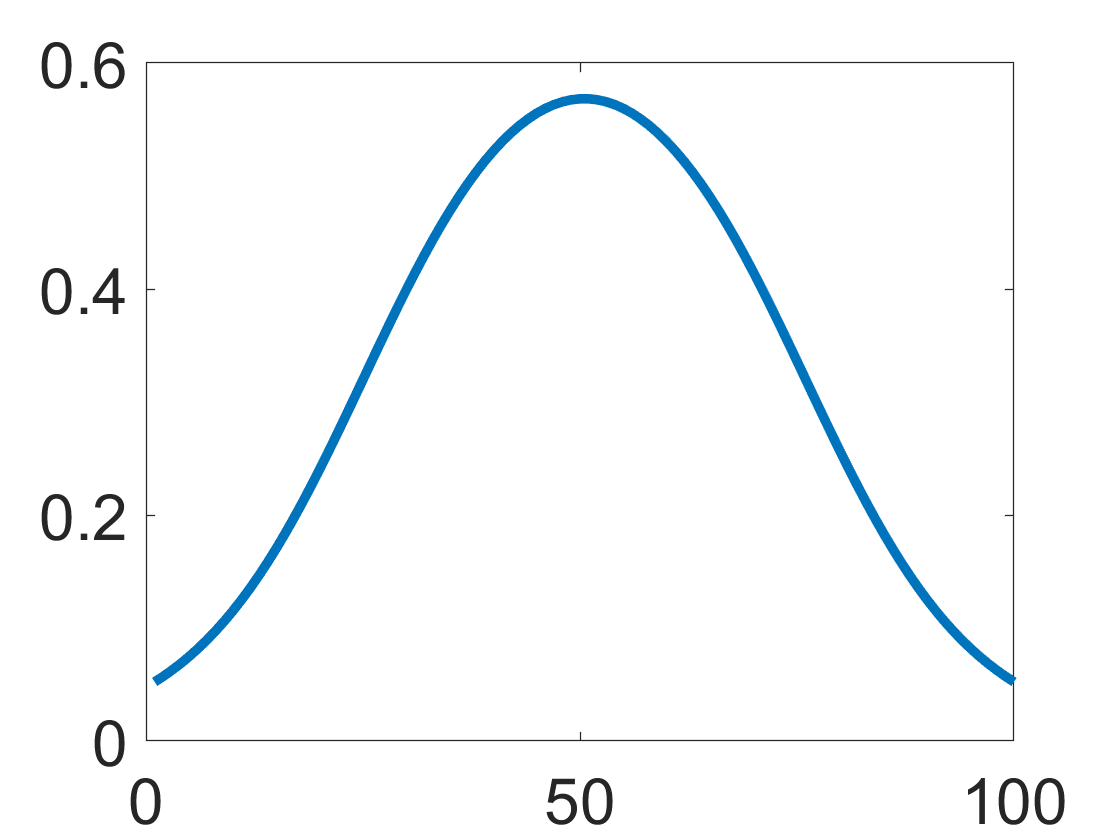}\\
    {\normalsize{(a) \texttt{phillips}}}&
    {\normalsize{(b) \texttt{gravity}}}&
    {\normalsize{(c) \texttt{shaw}}}
    \end{tabular}
    \caption{Row-block norms $\|A_{i}\|$ as a function of the block index $i$ for (a) \texttt{phillips}, (b) \texttt{gravity}, and (c) \texttt{shaw}.}
    \label{fig:A_ik}
\end{figure}

\begin{table}[htp!]
  \centering
  \setlength{\tabcolsep}{3pt}
  \begin{threeparttable}
  \caption{The comparison under discrepancy principle for \texttt{shaw}.}\label{tab:shaw}
    \begin{tabular}{cccccccccccc}
    \toprule
    \multicolumn{1}{c}{Method}&
    \multicolumn{2}{c}{DSBGD-S1}&
    \multicolumn{2}{c}{DSBGD-S2}&
    \multicolumn{2}{c}{DSBGD-S2}&
    \multicolumn{2}{c}{RBCD}&
    \multicolumn{2}{c}{LM}\\
    \cmidrule(lr){2-3} \cmidrule(lr){4-5} \cmidrule(lr){6-7} \cmidrule(lr){8-9}\cmidrule(lr){10-11}
    
    $\epsilon$ & $k_*$ & $e_*$ & $k_*$ & $e_*$ &$ k_*$ & $e_*$ & $k_*$ & $e_*$ & $k_*$ & $e_*$\\
    \midrule
        1e-3   &  3450.8 &   1.74e-1  &  342.3	&  2.21e-1   & 1862.8 &  1.10e-1  &   14711.6  &  4.37e-1     & 14449.0  & 7.29e-2 \cr
        5e-3   &  58.4   &   2.29e-1  &  29.4	&  2.40e-1   & 142.7  &  1.42e-1  &   372.3    &  4.71e-1     & 244.0      & 1.66e-1 \cr
        1e-2   &  36.3   &   2.30e-1  &  8.4    &  2.48e-1   & 43.4   &  1.53e-1  &   156.5    &  4.62e-1     & 154.0      & 1.71e-1 \cr
        5e-2   &  7.3    &   2.88e-1  &  1.7	&  2.80e-1   & 8.5    &  1.98e-1  &   31.5     &  4.94e-1     & 31.0       & 2.37e-1 \cr
    \bottomrule
    \end{tabular}
    \end{threeparttable}
\end{table}
~\\
~\\
~\\
\noindent\emph{\textbf{Example 5.2.}} 
In this example, we consider the problem \eqref{eqn:lininv} with the forward operator $A$ described in Remark \ref{rem:tenosr_product_A}.
This particular form of $A$ satisfies Assumption \ref{ass:A}, which is required for deriving the convergence rate in Theorem \ref{thm:err-total}.
Consider the integration operator \(K:L^{2}([0,1])\to L^{2}([0,1])\) defined by
\begin{equation}
K(x):[0,1]\to \mathbb{R}:s\mapsto \int_{0}^{s}x(t)\,\mathrm{d}t.
\label{eq:K-def}
\end{equation}
According to the Cauchy--Schwarz inequality, we have
\begin{equation}
\|K(x)\|^{2}
= \int_{0}^{1}\left(\int_{0}^{s}x(t)\,\mathrm{d}t\right)^{2}\mathrm{d}s
\leq \int_{0}^{1}s\int_{0}^{1}x(t)^{2}\,\mathrm{d}t\,\mathrm{d}s
= \frac{1}{2}\|x\|^{2}
\label{eq:K-bound}
\end{equation}
for all $x\in L^{2}([0,1])$, which shows that \(K\) is a well-defined linear bounded operator.
Using the operator \(K\), we consider the following forward model applied to a vector of functions \((x[i])_{i=1}^{n_x}\in (L^{2}([0,1]))^{n_x}\).

\begin{definition}[Forward operator]
\label{def:forward-operator}
For \(n_y\geq n_x\geq 1\) and a given matrix \(V=(v_{ij})_{i,j}\in\mathbb{R}^{n_y\times n_x}\) of rank \(n_y\), we define the forward operator
\begin{equation}
A:(L^{2}([0,1]))^{n_x}\to (L^{2}([0,1]))^{n_y}:
x\mapsto \left(\sum_{j=1}^{n_x}v_{ij}\,K(x[j])\right)_{i=1}^{n_y}.
\label{eq:A-def}
\end{equation}
\end{definition}

According to our general notation (cf. Remark \ref{rem:tenosr_product_A}), we have $A=V\otimes K$, and the theory presented in the previous sections can be applied for solving the inverse problem \(Ax=y\). Note that this equation is clearly ill-posed because the range of \(K\) is non-closed (and equal to the Sobolev space
$H_{0}^{1}([0,1]):=\bigl\{y\in L^{2}([0,1])\mid y'\in L^{2}([0,1])\ \wedge\ y(0)=0\bigr\}$
of all weakly differentiable functions vanishing at \(0\)).
More generally, one could replace the integration operator by any bounded (integral) operator \(K:L^{2}([0,1])\to L^{2}([0,1])\) with non-closed range.

For all presented numerical results we use $n_y = n_x = 30$ and take $V$ to be a random column-orthonormal matrix generated by the QR factorization of a standard Gaussian matrix, i.e., $V = Q$ where $[Q,\sim] = \operatorname{qr}(\operatorname{randn}(n_y, n_x))$. We discretize $K$ with the composite trapezoidal rule using $p=30$ intervals. The true unknown $x^{\dagger} = (x^{\dagger}[1], \ldots, x^{\dagger}[n
_x])$ is chosen such that only the first block is nonzero and is a piecewise constant $x^\dagger[1] = 0.2\chi_{[0,0.25)} + 0.8\chi_{[0.25,0.5)} + 0.4\chi_{[0.5,0.75)} + 0.9\chi_{[0.75,1]}$, while all other blocks are identically zero. 
The operator $A$ is discretized into a linear problem of the form \eqref{eqn:lininv} of the size $n_y = n_x = 30$ with $A_{i,j}: \mathbb{R}^{p^2} \rightarrow \mathbb{R}^{p^2}$ for any $1\leq i\leq n_y$ and $1\leq j\leq n_x$. 
The exact data \(y = Ax^{\dagger}\) are computed via numerical integration followed by application of $V$ and the noisy data $y^\delta$ is generated in the same way as \eqref{eqn:num_y_data}. 

Now we examine more closely the convergence behavior of the DSBGD iterates, and compare it against that of RBCD, SGD and LM methods. All iterative methods are initialized to zero. DSBGD and SGD methods use a constant stepsize $\eta_k = c_0\max_{1\leq i\leq n_y}\|A_i\|^{-2}$ with $c_0 = 1.8$, and RBCD and LM use a constant stepsize with $\eta_k = c_0\|A\|_F^{-2}$ with $c_0 = 1.99$. The constant stepsize is chosen to be as large as possible while still ensuring stable reconstruction. All statistical quantities generated by DSBGD, RBCD and SGD methods are computed from $50$ runs.

The accuracy of the reconstructions is measured by the square root of the relative mean squared errors, defined by 
\begin{equation*}
    e_* = \mathbb{E}[ \| x_{k_*}^{\delta} - x^{\dagger} \|^2 ]^{\frac12} / \| x^{\dagger} \|,\quad 
    e_{\min} = \mathbb{E}[ \| x_{\hat{k}}^{\delta} - x^{\dagger} \|^2 ]^{\frac12} / \| x^{\dagger} \|,
\end{equation*}
where the stopping index $k_* = k(\delta)$, measured in terms of epoch count, is selected by the discrepancy principle \eqref{eqn:discrepancy_1} with $\tau=1.1$ for all methods, $\hat{k}$ denotes the iteration index attaining the smallest error along the respective iteration trajectory for DSBGD, RBCD  and SGD methods. 
The relative error for LM is given by $\| x_{k_*}^{\delta} - x^{\dagger} \|/\| x^{\dagger} \|$.
Here, one epoch refers to one Landweber iteration, $n_x n_y$ DSBGD iterations and $n_x$ RBCD iterations or $n_y$ SGD iterations.

\begin{table}[htp!]
  \centering
  \small
  \setlength{\tabcolsep}{3pt}
  \begin{threeparttable}
  \caption{The comparison between DSBGD, RBCD, SGD and LM for \textbf{Example 5.2}.}\label{tab:num_integral}
    \begin{tabular}{ccccccccccccc}
    \toprule
    \multicolumn{1}{c}{Method}&
    \multicolumn{3}{c}{DSBGD}&
    \multicolumn{3}{c}{RBCD}&
    \multicolumn{3}{c}{SGD}&
    \multicolumn{3}{c}{LM}\\
    \cmidrule(lr){2-4} \cmidrule(lr){5-7} \cmidrule(lr){8-10} \cmidrule(lr){11-13}
       $\epsilon$ & $k_{*}$ & $e_{*}$ & $e_{\min}$ & $k_{*}$ & $e_{*}$ & $e_{\min}$ & $k_{*}$ & $e_{*}$ & $e_{\min}$ & $k_{*}$ & $e_{*}$ & $e_{\min}$ \\
    \midrule
        1e-3   &  397.0 &   1.69e-1 & 1.38e-1 & 13068.8   &  1.69e-1  & 1.38e-1 &  397.0 &  1.96e-1 & 1.38e-1 & 10350.0 & 1.69e-1 & 1.38e-1\cr
        5e-3   &  61.8   &   2.64e-1 & 2.35e-1 & 1994.0  	& 2.64e-1   & 2.35e-1 &  61.4  &  2.64e-1 & 2.35e-1 & 1996.0   & 2.65e-1 & 2.35e-1\cr
        1e-2   &  32.9    &   3.29e-1 & 2.83e-1 & 1040.4	&  3.29e-1  & 2.82e-1 &  33.0  &  3.28e-1 & 2.83e-1 & 1051.0   & 3.30e-1 & 2.82e-1\cr
        5e-2   &  8.0     &   5.15e-1 & 4.83e-1 & 107.3	    &  5.73e-1  & 4.80e-1 &   10.9  &  4.99e-1 & 4.85e-1 &  109.0   & 5.73e-1 & 4.81e-1\cr
    \bottomrule
    \end{tabular}
    \end{threeparttable}
\end{table}

The numerical results with different noise levels are shown in Table \ref{tab:num_integral}, and the convergence behavior of the relative error $e$ versus the iteration epoch $k$ is illustrated in Figure \ref{fig:num_e}. For all noise levels, four methods achieve comparable accuracy, indicating order-optimal convergence for all of them. Moreover, both $e_*$ and $e_{\min}$ for all methods decrease toward zero as the relative noise level $\epsilon$ tends to zero, while the stopping indices increase. Compared with RBCD and LM methods, our proposed DSBGD method converges as fast as SGD. 
A distinct advantage of DSBGD emerges when the discrepancy principle is adopted as the stopping rule. 
As described in Algorithm \ref{alg:DSBGD}, the residual in DSBGD is updated recursively,  $r_{k+1}^\delta 
=r_{k}^\delta-\eta_k A_{\cdot,j_k} A_{i_k,j_k}^*r_{k,i_k}^\delta$. In contrast, SGD must evaluate the full operator to update the residual at each iteration. Consequently, DSBGD requires less computational effort when the discrepancy principle is used for termination.

\begin{figure}[hbt!]
    \centering
    \setlength{\tabcolsep}{0pt}
    \begin{tabular}{cccc}
    \includegraphics[height=.19\textwidth]{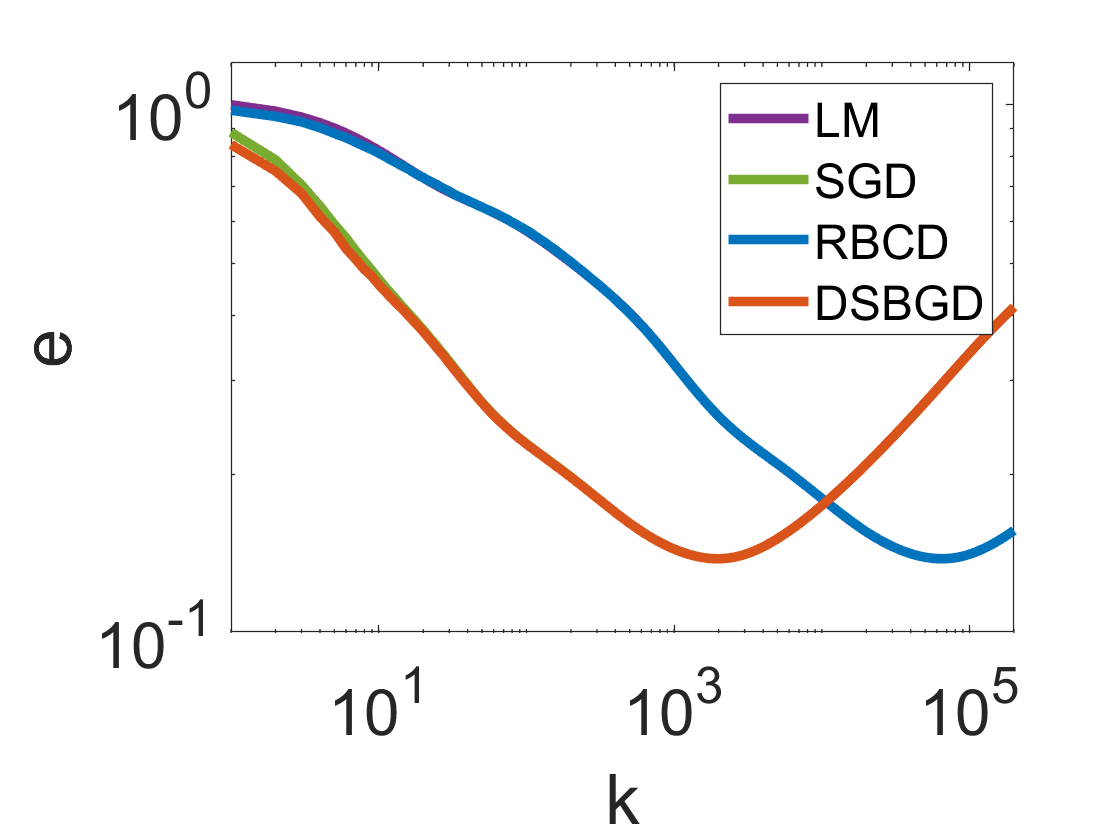}&
    \includegraphics[height=.19\textwidth]{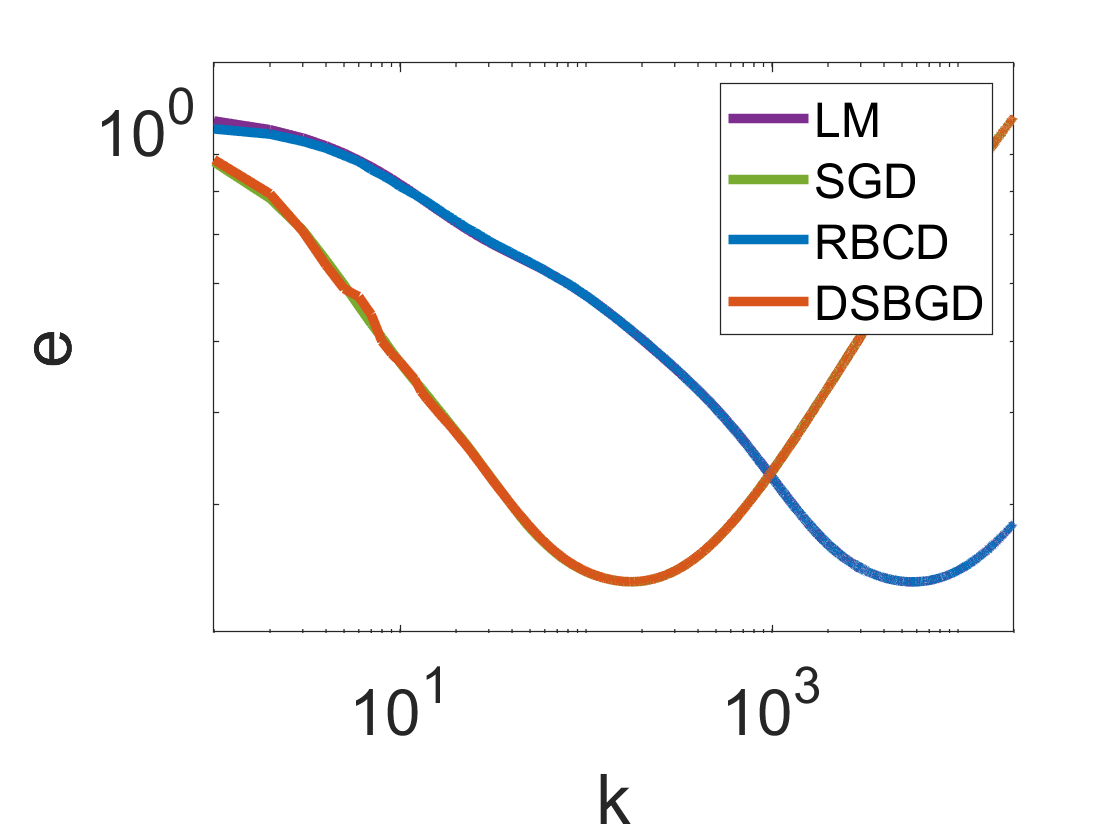}&
    \includegraphics[height=.19\textwidth]{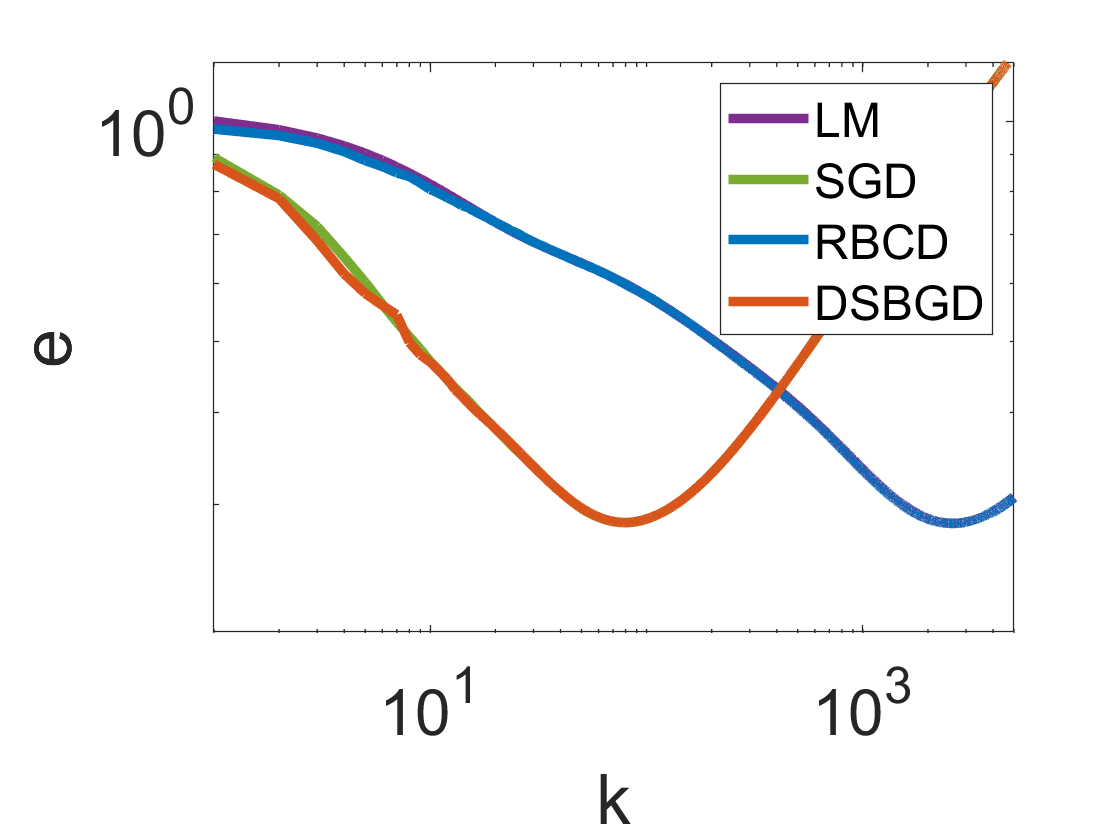}&
    \includegraphics[height=.19\textwidth]{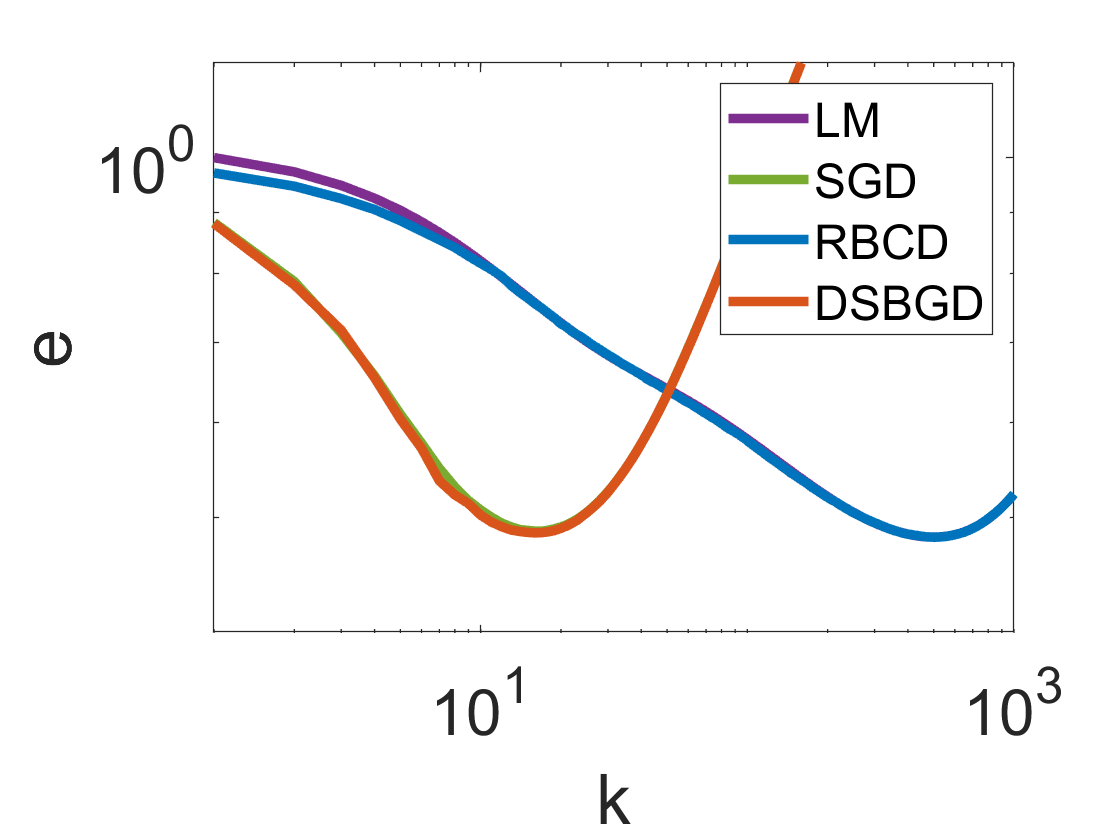}\\
    {\normalsize{(a) $\epsilon = 1e-3$}}&
    {\normalsize{(b) $\epsilon = 5e-3$}}&
    {\normalsize{(c) $\epsilon = 1e-2$}}&
    {\normalsize{(d) $\epsilon = 5e-2$}}
    \end{tabular}
    \caption{The convergence of the relative error $e$ versus the iteration epoch $k$.}
    \label{fig:num_e}
\end{figure}

\appendix
\section{Auxiliary estimates}\label{app:estimate}
In this appendix, we collect a set of supplementary estimates and lengthy technical proofs of several results. 
To establish the finite-iteration termination property under the discrepancy principle \eqref{eqn:discrepancy_1} and derive the convergence rates under \textit{a priori} stopping rules by spectral decomposition, we present several preliminary estimates in the following lemmas.
We first gather several useful estimates from \cite{JinZhouZou:2020,ZhouDSGD:2024}.
\begin{lemma}{\cite[Lemma A.5] {ZhouDSGD:2024}}\label{lem:estimate-B}
For any $\ell<k$, $s\geq 0$ and $Z=X$ or $Y$, there holds
\begin{equation*} \bigg\|\prod_{t=\ell}^kP^Z_tB_Z^{s}\bigg\|=\|\Pi_{t=\ell}^k(I_Z-\eta_t B_Z)B_Z^{s}\| 
\leq \Big(se^{-1}(\sum_{t=\ell}^k\eta_t)^{-1}\Big)^s.
\end{equation*}
Further, under Assumptions \ref{ass:stepsize}(iii), there holds
\begin{equation*} \bigg\|\prod_{t=\ell}^kP^Z_tB_Z^{s}\bigg\|\leq \Big(\frac{sL}{e c_0}\Big)^s(k-\ell+1)^{-s}k^{\alpha s}.
\end{equation*}
\end{lemma}
\begin{proof}
By the definition of $P^Z_t$, the first assertion in the lemma follows directly from \cite[Lemma A.5]{ZhouDSGD:2024} with $B_F=B_Z$ and $B_G=O_Z$. 
Then Assumption \ref{ass:stepsize}(iii) implies
\begin{align*}
\sum_{t=\ell}^k\eta_t= c_0 L^{-1}\sum_{t=\ell}^kt^{-\alpha}
\geq c_0 L^{-1}(k-\ell+1)k^{-\alpha},
\end{align*}
and hence
\begin{align*}
\bigg\|\prod_{t=\ell}^kP^Z_tB_Z^{s}\bigg\|\leq
\Big(se^{-1}\big(c_0 L^{-1}(k-\ell+1)k^{-\alpha}\big)^{-1}\Big)^s
\leq \Big(\frac{sL}{e c_0}\Big)^s(k-\ell+1)^{-s}k^{\alpha s}.
\end{align*}
This complete the proof of the lemma.
\end{proof}

\begin{lemma}{\cite[Lemma A.6] {ZhouDSGD:2024}}\label{lem:sum}
If $\eta_t=c_0 L^{-1} t^{-\alpha}$, $\alpha\in(0,1)$ and $\beta\in[0,1]$, then there hold
\begin{align}
&\;\;\;\sum_{\ell=1}^{k-1} \frac{\eta_\ell}{(\sum_{t=\ell+1}^k\eta_t)^{\frac12}} \leq
c_0^{\frac12} L^{-\frac12}B(\tfrac12,1-\alpha) k^{\frac{1-\alpha}{2}},\label{eqn:bdd-phi-r}\\
&\;\;\;\sum_{\ell=1}^{[\frac k2]}\frac{\eta_\ell^2}{\sum_{t=\ell+1}^k\eta_t}\ell^{-\beta} \leq c_{\alpha,\beta} k^{-(1-\alpha)+\max(0,1-2\alpha-\beta)},\label{eqn:bdd-phi-r-2}\\
&\sum_{\ell=[\frac k2]+1}^{k-1}\frac{\eta_\ell^2}{\sum_{t=\ell+1}^k\eta_t}\ell^{-\beta}\leq c'_{\alpha,\beta} k^{-(\alpha+\beta)}\ln k,\label{eqn:bdd-phi-r-22}
\end{align}
where we slightly abuse $k^{-\max(0,0)}$ for $\ln k$, $B(\cdot,\cdot)$ denotes the Beta function defined by
\begin{equation*}
B(a,b)=\int_0^1s^{a-1}(1-s)^{b-1}{\rm d} s \;\;\mbox{ for any }a,b>0,
\end{equation*} 
and the constants $c'_{\alpha,\beta}=2^{2\alpha+\beta+1} c_0 L^{-1}$ and 
\begin{align*}
c_{\alpha,\beta} &=2 c_0 L^{-1}\left\{\begin{array}{ll}
1+(2\alpha+\beta-1)^{-1},   & 2\alpha+\beta>1,\\
2,                  & 2\alpha+\beta=1,\\
2^{2\alpha+\beta-1}(1-2\alpha-\beta)^{-1},     & 2\alpha+\beta<1.
\end{array}\right.
\end{align*}
\end{lemma}

\begin{lemma}\label{lem:N}
Let $R: X\rightarrow X$ be a deterministic bounded linear operator. Then for any $t\geq 0$, there holds
\begin{align*}
\E[\|R N_t r_t^\delta\|^2|\mathcal{F}_t]\leq n^{-1}L \|R\|^2 \|r_t^\delta\|^2.
\end{align*}
Further, if Assumption \ref{ass:A} holds,
there holds
\begin{align*}
\E[\|R N_t r_t^\delta\|^2|\mathcal{F}_t]\leq C_A \|R B_X^\frac12\|^2 \|r_t^\delta\|^2.
\end{align*}
\end{lemma}
\begin{proof}
The definition $N_t= n^{-1}A^*- D_{j_t}^{n_x} A^* D_{i_t}^{n_y}$, the identity $ n^{-1}A^*=\E[D_{j_t}^{n_x} A^* D_{i_t}^{n_y}|\mathcal{F}_t]$ and the bias-variance decomposition imply
\begin{align}
\E[\|R N_t r_t^\delta\|^2|\mathcal{F}_t]&=\E[\|R (n^{-1}A^*- D_{j_t}^{n_x} A^* D_{i_t}^{n_y}) r_t^\delta\|^2|\mathcal{F}_t]\nonumber\\
&=\E[\|R D_{j_t}^{n_x} A^* D_{i_t}^{n_y} r_t^\delta\|^2|\mathcal{F}_t]-\|n^{-1} R A^* r_t^\delta\|^2 \label{eqn:Nr_1}\\
&\leq n_x^{-1}\|R\|^2 \E[\|A^* D_{i_t}^{n_y}\|^2\|D_{i_t}^{n_y}r_t^\delta\|^2|\mathcal{F}_t]
\leq n^{-1}L \|R\|^2 \|r_t^\delta\|^2.\nonumber
\end{align}
Building on the estimate \eqref{eqn:Nr_1}, we can also bound $\E[\|R N_t r_t^\delta\|^2|\mathcal{F}_t]$ by
\begin{align*}
\E[\|R N_t r_t^\delta\|^2|\mathcal{F}_t]
&\leq \E[\|R D_{j_t}^{n_x} A^*\|^2 \|D_{i_t}^{n_y} r_t^\delta\|^2|\mathcal{F}_t]
\leq n_y^{-1}\E[\|R D_{j_t}^{n_x} A^*\|^2 |\mathcal{F}_t]\|r_t^\delta\|^2.
\end{align*}
Now, we claim that, under Assumption \ref{ass:A}, for any $j=1,\cdots,n_y$, there exists a bounded linear operator $T_j :Y \rightarrow Y$ such that $\|T_j\|\leq \sqrt{c_j}$ and $D_{j}^{n_x} A^*=A^* T_j$.
In fact, according to Douglas’ lemma, for bounded operators $D_{j}^{n_x} A^*$ and $A^*$, $T_j$ exists if and only if ${\rm range}(D_{j}^{n_x} A^*)\subseteq {\rm range}(A^*)$ or 
\begin{align*}
D_{j}^{n_x} A^* A D_{j}^{n_x} =(D_{j}^{n_x} A^*)(D_{j}^{n_x} A^*)^* \preceq c_j (A^*)(A^*)^*=c_j A^* A.
\end{align*}
Consequently, with the definitions $B_X=n^{-1}A^* A$ and $C_A=\sum_{j=1}^{n_x}c_i$, there holds
\begin{align*}
\E[\|R N_t r_t^\delta\|^2|\mathcal{F}_t]
&\leq n_y^{-1}\E[\|R A^* T_{j_t}\|^2 |\mathcal{F}_t]\|r_t^\delta\|^2
\leq n^{-1} \sum_{j=1}^{n_x}c_i\|R A^*\|^2 \|r_t^\delta\|^2 
= C_A \|R B_X^\frac12\|^2 \|r_t^\delta\|^2.
\end{align*}
This completes the proof of the lemma.
\end{proof}

\subsection{Proof of Lemma \ref{lem:r_k^delta}}\label{app:lem:r_k^delta}

In this part, we provide the proof of Lemma \ref{lem:r_k^delta}.
By the definitions of $P^Y_k$, $N_k$ and $r_k^\delta=Ae_k^\delta -\xi$, it follows from the recursion \eqref{eqn:e_k^delta} that
\begin{align}\label{eqn:recursion_r_0}
r_{k+1}^\delta
= r_k^\delta- \eta_k A D_{j_k}^{n_x} A^* D_{i_k}^{n_y} r_k^\delta
= P^Y_k r_k^\delta +\eta_k A N_k r_k^\delta.
\end{align}
Iterating the above recurrence gives
\begin{align}\label{eqn:recursion_r}
r_{k+1}^\delta
= \prod_{t=k-1}^k P^Y_t r_{k-1}^\delta +\sum_{\ell=k-1}^k \eta_\ell \prod_{t=\ell+1}^k P^Y_t A N_\ell r_\ell^\delta
= \prod_{t=1}^k P^Y_t r_{1}^\delta +\sum_{\ell=1}^k \eta_\ell \prod_{t=\ell+1}^k P^Y_t A N_\ell r_\ell^\delta,
\end{align}
with the initial residual $r_1^\delta=Ae_1^\delta-\xi$ and the convention $\prod_{t=k+1}^k P^Y_t=I_Y$.
Then, by the bias-variance decomposition, the Cauchy-Schwarz inequality and the identity $\E[\langle N_{t}r_t^\delta, N_{\ell}r_\ell^\delta\rangle|\mathcal{F}_t]=0$ for any $t>\ell$, we can bound the residual by
\begin{align}\label{eqn:residual}
\E[\|r_{k+1}^\delta\|^2]
&=\E\bigg[\bigg\|\prod_{t=1}^k P^Y_t Ae_1^\delta-\prod_{t=1}^k P^Y_t\xi +\sum_{\ell=1}^k \eta_\ell \prod_{t=\ell+1}^k P^Y_t A N_\ell r_\ell^\delta\bigg\|^2\bigg]\nonumber\\
&\leq 2\bigg\|\prod_{t=1}^k P^Y_t\xi\bigg\|^2+2\bigg\|\prod_{t=1}^k P^Y_t Ae_1^\delta\bigg\|^2+\sum_{\ell=1}^k\eta_\ell^2\E\bigg[\bigg\|  \prod_{t=\ell+1}^k P^Y_t A N_\ell r_\ell^\delta\bigg\|^2\bigg] \leq 2\delta^2 + 2\mathfrak{I}_1+ \mathfrak{I}_2,
\end{align}
where $\mathfrak{I}_1=\big\|\prod_{t=1}^k P^Y_t Ae_1^\delta\big\|^2\leq n\big\|\prod_{t=1}^k P^Y_t B_Y^\frac12\big\|^2\|e_1^\delta\|^2$ 
and
$\mathfrak{I}_2=\sum_{\ell=1}^k\eta_\ell^2\E\big[\big\|  \prod_{t=\ell+1}^k P^Y_t A N_\ell r_\ell^\delta\big\|^2\big]$. 
Note that $\mathfrak{I}_2=\E[\|r_k^\delta-\E[r_k^\delta]\|^2]$.
Next, Lemmas \ref{lem:estimate-B} and \ref{lem:N} and the inequality $k+1\leq 2k$ yield
\begin{align}
\mathfrak{I}_1&\leq n\bigg\|\prod_{t=1}^k P^Y_t B_Y^\frac12\bigg\|^2\|e_1^\delta\|^2
\leq \frac{nL\|e_1^\delta\|^2}{2e c_0}k^{-(1-\alpha)}
\leq \frac{nL\|e_1^\delta\|^2}{2^{\alpha} e c_0}(k+1)^{-(1-\alpha)},\label{eqn:I_1}\\
\mathfrak{I}_2&\leq n^{-1}L\sum_{\ell=1}^k\eta_\ell^2\bigg\|  \prod_{t=\ell+1}^k P^Y_t A\bigg\|^2 \E[\|r_\ell^\delta\|^2]
=L\sum_{\ell=1}^k\eta_\ell^2\bigg\|  \prod_{t=\ell+1}^k P^Y_t B_Y^\frac12\bigg\|^2 \E[\|r_\ell^\delta\|^2]:=\mathfrak{I}'_2\label{eqn:I_2_for_rate}\\
&\leq \frac{\|A\|^2L}{n} \eta_k^2\E[\|r_k^\delta\|^2]+\frac{L}{2e} \sum_{\ell=1}^{k-1}\frac{\eta_\ell^2}{\sum_{t=\ell+1}^k\eta_t}\E[\|r_\ell^\delta\|^2]\nonumber.
\end{align}
The rest of the proof proceeds by mathematical induction.
For the case $k=1$, $\|r_1^\delta\|\leq \|Ae_1^\delta\|+\delta$, and thus the estimate \eqref{eqn:r_k^delta} holds for any sufficiently large $c_1^*\geq n^{-1}\|Ae_1^\delta\|^2$ and $c_2^{*}\geq 0$. 
Now, we assume that the bound
holds up to case with sufficiently large $k$, and prove the assertion for case $k+1$.
Under Assumption \ref{ass:stepsize}(iii), with the estimate $\E[\|r_\ell^\delta\|^2]\leq nc_1^* \ell^{-(1-\alpha)}+nc_2^*\delta^2$ for any $\ell\leq k$, the inequality $\|A\|^2=\sum_{i=1}^{n_y}\|A_i\|^2\leq n_yL$ and Lemma \ref{lem:sum}, we derive
\begin{align}\label{eqn:I_2}
\mathfrak{I}_2&\leq \mathfrak{I}'_2 \leq \frac{c_0^2\|A\|^2}{nL}  k^{-2\alpha} \E[\|r_k^\delta\|^2]+\frac{L}{2e} \sum_{\ell=1}^{k-1}\frac{\eta_\ell^2}{\sum_{t=\ell+1}^k\eta_t}\E[\|r_\ell^\delta\|^2]\nonumber\\
&\leq \frac{c_0^2}{n_x}  k^{-2\alpha} (nc_1^* k^{-(1-\alpha)}+nc_2^*\delta^2)
+\frac{L}{2e} \sum_{\ell=1}^{k-1}\frac{\eta_\ell^2}{\sum_{t=\ell+1}^k\eta_t}(nc_1^* \ell^{-(1-\alpha)}+nc_2^* \delta^2)\nonumber\\
&\leq \frac{c_0^2}{n_x}   (nc_1^* k^{-(1+\alpha)}+nc_2^*k^{-2\alpha}\delta^2)
+\frac{L}{2e}(nc_1^* \mathfrak{I}_{2,1}+nc_2^*\delta^2 \mathfrak{I}_{2,2}),
\end{align}
where 
\begin{align}
\mathfrak{I}_{2,1}&=\sum_{\ell=1}^{k-1}\frac{\eta_\ell^2}{\sum_{t=\ell+1}^k\eta_t}\ell^{-(1-\alpha)}
\leq
2c_0L^{-1}(1+\alpha^{-1})k^{-(1-\alpha)}+2^{2+\alpha}c_0L^{-1}k^{-1}\ln k,\label{eqn:I_21}\\
\mathfrak{I}_{2,2}&=\sum_{\ell=1}^{k-1}\frac{\eta_\ell^2}{\sum_{t=\ell+1}^k\eta_t}
\leq 2^{1+2\alpha}c_0 L^{-1} k^{-\alpha}\ln k\nonumber\\
&\qquad\qquad\qquad\qquad\;\; + c_0L^{-1} 
\begin{cases}
2^2 k^{-\frac12}\ln k, & \mbox{when}\; \alpha = \frac12,\\
\max\big(\tfrac{2^{2\alpha}}{1-2\alpha}k^{-\alpha},2(1+\tfrac{1}{2\alpha-1})k^{-(1-\alpha)}\big), & \mbox{otherwise}.
\end{cases}
\label{eqn:I_22}
\end{align}
Then, by using the inequality, for any $s>0$ and $k\geq 1$,
\begin{equation}\label{eqn:bdd-log}
k^{-s}\ln k \leq (es)^{-1},
\end{equation}
we can further bound $\mathfrak{I}_{2,i}$ by
\begin{align}
\mathfrak{I}_{2,1}&
\leq 2\big(1+\alpha^{-1}+2^{1+\alpha}k^{-\alpha}\ln k\big)c_0L^{-1}k^{-(1-\alpha)}
\leq 2\big(1+(1+2^{1+\alpha}e^{-1})\alpha^{-1}\big)c_0L^{-1}k^{-(1-\alpha)},\label{eqn:I_21_2}\\
\mathfrak{I}_{2,2}&
\leq 2^{1+2\alpha}c_0 L^{-1} (k^{-\frac{\alpha}{2}}\ln k)k^{-\frac{\alpha}{2}}+ c_0L^{-1} 
\begin{cases}
2^2 (k^{-\frac14}\ln k)k^{-\frac14}, & \mbox{when}\; \alpha = \frac12\\
\max\big(\tfrac{2^{2\alpha}}{1-2\alpha}k^{-\alpha},2(1+\tfrac{1}{2\alpha-1})k^{-(1-\alpha)}\big), & \mbox{otherwise}
\end{cases}\nonumber\\
&\leq 2^{2+2\alpha}e^{-1}c_0 \alpha^{-1} L^{-1} k^{-\frac{\alpha}{2}}+ c_0L^{-1} 
\begin{cases}
2^4 e^{-1} k^{-\frac14}, & \mbox{when}\; \alpha = \frac12,\\
\max\big(\tfrac{2^{2\alpha}}{1-2\alpha}k^{-\alpha},2(1+\tfrac{1}{2\alpha-1})k^{-(1-\alpha)}\big), & \mbox{otherwise}.
\end{cases}
\nonumber
\end{align}
Together with the identity $k+1=\frac{k+1}{k}k$, for any 
sufficiently large $k$, $c_1^*$ and $c_2^*$, and $c_0 \leq \alpha/2$ which implies
$$c_0\leq\alpha/2=\frac{e}{e/2}\big(4\alpha^{-1}\big)^{-1}<\frac{e}{e/2}\big((2+2^{2}e^{-1})\alpha^{-1}\big)^{-1}<\frac{e}{e/2}\big(1+(1+2^{1+\alpha}e^{-1})\alpha^{-1}\big)^{-1},$$ 
it follows from \eqref{eqn:I_2} that
\begin{align}\label{eqn:I_2_f}
\mathfrak{I}_2\leq \tfrac{2}{e} n c_1^* (k+1)^{-(1-\alpha)}+\tfrac12 n c_2^* \delta^2.
\end{align}
Then substituting the estimates \eqref{eqn:I_1} and \eqref{eqn:I_2_f} into the inequality \eqref{eqn:residual} yields
\begin{align*}
\E[\|r_{k+1}^\delta\|^2]
&\leq 2\delta^2 + 2\mathfrak{I}_1+ \mathfrak{I}_2
\leq 2\delta^2+\frac{2^{1-\alpha}nL\|e_1^\delta\|^2}{e c_0}(k+1)^{-(1-\alpha)}+ \tfrac{2}{e} n c_1^* (k+1)^{-(1-\alpha)}+\tfrac12 n c_2^* \delta^2\\
&\leq n c_1^* (k+1)^{-(1-\alpha)}+n c_2^* \delta^2,
\end{align*}
for sufficiently large $c_1^*\geq \frac{2^{1-\alpha}L\|e_1^\delta\|^2}{(e-2) c_0}$ and $c_2^*\geq \frac{4}{n}$.
This gives the estimate \eqref{eqn:r_k^delta}.

Next, we discuss the mean residual
$\|\E[r_{k+1}^\delta]\|$. With the recursion \eqref{eqn:recursion_r}, the triangle inequality and the estimate \eqref{eqn:I_1}, we obtain
\begin{align*}
\|\E[r_{k+1}^\delta]\|
&= \bigg\|\prod_{t=1}^k P^Y_t r_{1}^\delta\bigg\|= \bigg\|\prod_{t=1}^k P^Y_t (Ae_{1}^\delta-\xi)\bigg\|
\leq \bigg\|\prod_{t=1}^k P^Y_t Ae_{1}^\delta\bigg\|+\bigg\|\prod_{t=1}^k P^Y_t \xi\bigg\|\\
&\leq \sqrt{\mathfrak{I}_1}+\delta
\leq \sqrt{\tfrac{nL}{2^\alpha e c_0}}\|e_1^\delta\|(k+1)^{-\frac{1-\alpha}{2}}+\delta,
\end{align*}
which yields the estimate \eqref{eqn:E[r_k^delta]}.
Finally, we derive the bound of $\mathfrak{I}_2=\E[\|r_k^\delta-\E[r_k^\delta]\|^2]$ in \eqref{eqn:r_k^delta_var} from \eqref{eqn:I_2}, \eqref{eqn:I_22} and \eqref{eqn:I_21_2}.
It follows directly from \eqref{eqn:I_22} and the inequality \eqref{eqn:bdd-log} that
\begin{align*}
\mathfrak{I}_{2,2}&
\leq c_0 L^{-1}k^{-\frac{\alpha}{2}}\left(2^{3} (k^{-\frac{\alpha}{4}}\ln k)k^{-\frac{\alpha}{4}} +
\begin{cases}
2^2 (k^{-\frac18}\ln k)k^{-\frac18}, & \mbox{when}\; \alpha = \frac12\\
\max(\tfrac{2}{1-2\alpha}k^{-\frac{\alpha}{2}},\tfrac{2^2}{2\alpha-1}k^{-(1-\frac32\alpha)}),  &\mbox{otherwise}  
\end{cases}
\right)\\
&\leq c_0 L^{-1}k^{-\frac{\alpha}{2}}\left(2^{4} \alpha^{-1}k^{-\frac{\alpha}{4}}+
\begin{cases}
2^4 k^{-\frac18}, & \mbox{when}\; \alpha = \frac12\\
\max(\tfrac{2}{1-2\alpha}k^{-\frac{\alpha}{2}},\tfrac{2^2}{2\alpha-1}k^{-(1-\frac32\alpha)}),  &\mbox{otherwise}  
\end{cases}
 \right).
\end{align*}
Then with the identity $k+1=\frac{k+1}{k}k$, for any $\alpha\in(0,\frac23)$, 
$c_0 \leq \alpha/2$, and sufficiently large $k$, $c_1^*$ and $c_2^*$, the decomposition \eqref{eqn:I_2} gives
\begin{align*}
\mathfrak{I}_2\leq n c_1^* (k+1)^{-(1-\alpha)}+n c_2^* (k+1)^{-\frac{\alpha}{2}}\delta^2.
\end{align*}
This completes the proof of the lemma.

\bibliographystyle{abbrv}
\bibliography{sgd}

@article{BrahmaKolbitsch:2023,
author = {Brahma, Sherine and Kolbitsch, Christoph and Martin, Joerg and Schaeffter, Tobias and Kofler, Andreas},
title = {Data-efficient Bayesian learning for radial dynamic MR reconstruction},
journal = {Medical Physics},
volume = {50},
number = {11},
pages = {6955-6977},
doi = {https://doi.org/10.1002/mp.16543},
url = {https://aapm.onlinelibrary.wiley.com/doi/abs/10.1002/mp.16543},
eprint = {https://aapm.onlinelibrary.wiley.com/doi/pdf/10.1002/mp.16543},
year = {2023}
}

@article{CaraccioloCaiafa:2022,
doi = {10.1088/2057-1976/ac835b},
url = {https://doi.org/10.1088/2057-1976/ac835b},
year = {2022},
publisher = {IOP Publishing},
volume = {8},
number = {6},
pages = {065010},
author = {Caracciolo, S F and Caiafa, C F and Martínez Pería, F D and Arini, P D},
title = {A fast algorithm for spatiotemporal signals recovery using arbitrary dictionaries with application to electrocardiographic imaging},
journal = {Biomedical Physics \& Engineering Express}
}

@article{Zhang:2017,
doi = {10.1088/1361-6560/aa7122},
url = {https://doi.org/10.1088/1361-6560/aa7122},
year = {2017},
month = {jun},
publisher = {IOP Publishing},
volume = {62},
number = {13},
pages = {5556},
author = {Zhang, Houjin and Zeng, Dong and Lin, Jiahui and Zhang, Hao and Bian, Zhaoying and Huang, Jing and Gao, Yuanyuan and Zhang, Shanli and Zhang, Hua and Feng, Qianjin and Liang, Zhengrong and Chen, Wufan and Ma, Jianhua},
title = {Iterative reconstruction for dual energy CT with an average image-induced nonlocal means regularization},
journal = {Physics in Medicine \& Biology}
}

@article{PerelliAndersen:2021,
    author = {Perelli, Alessandro and Andersen, Martin S.},
    title = {Regularization by denoising sub-sampled Newton method for spectral CT multi-material decomposition},
    journal = {Philosophical Transactions of the Royal Society A: Mathematical, Physical and Engineering Sciences},
    volume = {379},
    number = {2200},
    pages = {20200191},
    year = {2021},
    month = {05},
    issn = {1364-503X},
    doi = {10.1098/rsta.2020.0191},
    url = {https://doi.org/10.1098/rsta.2020.0191},
    eprint = {https://royalsocietypublishing.org/rsta/article-pdf/doi/10.1098/rsta.2020.0191/373300/rsta.2020.0191.pdf},
}

@ARTICLE{LiRavishankarLongFessler:2020,
  author={Li, Zhipeng and Ravishankar, Saiprasad and Long, Yong and Fessler, Jeffrey A.},
  journal={IEEE Transactions on Medical Imaging}, 
  title={DECT-MULTRA: Dual-Energy CT Image Decomposition With Learned Mixed Material Models and Efficient Clustering}, 
  year={2020},
  volume={39},
  number={4},
  pages={1223-1234},
  doi={10.1109/TMI.2019.2946177}}

@article{Nesterov：2012,
  title={Efficiency of Coordinate Descent Methods on Huge-Scale Optimization Problems},
  author={Yurii Nesterov},
  journal={SIAM J. Optim.},
  year={2012},
  volume={22},
  pages={341-362},
  url={https://api.semanticscholar.org/CorpusID:1424102}
}

@article{PoudelLouAnastasio:2019,
    doi = {10.1088/1361-6560/ab2017},
    url = {https://doi.org/10.1088/1361-6560/ab2017},
    year = {2019},
    publisher = {IOP Publishing},
    volume = {64},
    number = {14},
    pages = {14TR01},
    author = {Poudel, Joemini and Lou, Yang and Anastasio, Mark A},
    title = {A survey of computational frameworks for solving the acoustic inverse problem in three-dimensional photoacoustic computed tomography},
    journal = {Physics in Medicine \& Biology}}

@article {XiaWangHan:2022,
    AUTHOR = {Xia, Yuxin and Wang, Wei and Han, Bo},
     TITLE = {A fast averaged {K}aczmarz iteration with convex penalty for
              inverse problems in {H}ilbert spaces},
   JOURNAL = {SIAM J. Imaging Sci.},
  FJOURNAL = {SIAM Journal on Imaging Sciences},
    VOLUME = {15},
      YEAR = {2022},
    NUMBER = {3},
     PAGES = {1079--1103},
      ISSN = {1936-4954},
   MRCLASS = {65J20 (47A52 65J22)},
  MRNUMBER = {4452955},
MRREVIEWER = {Stefan\ Kindermann},
       DOI = {10.1137/21M1445181},
       URL = {https://doi.org/10.1137/21M1445181},
}

@article{RabanserNeumannHaltmeier:2019,
author = {Rabanser, Simon and Neumann, Lukas and Haltmeier, Markus},
title = {Analysis of the Block Coordinate Descent Method for Linear Ill-Posed Problems},
journal = {SIAM Journal on Imaging Sciences},
volume = {12},
number = {4},
pages = {1808-1832},
year = {2019},
doi = {10.1137/19M1243956},

URL = { 
    
        https://doi.org/10.1137/19M1243956
    
    

},
eprint = { 
    
        https://doi.org/10.1137/19M1243956
    
    

}
}

@article{BaoGuoXingLi:2026,
author = {Bao, Wendi and Guo, Zhiwei and Xing, Lili and Li, Weiguo},
title = {On stochastic block methods for solving nonlinear equations},
journal = {Numerical Algorithms},
year = {2026},
volume = {101},
number = {2},
pages = {915--931},
doi = {10.1007/s11075-025-02027-7},
issn = {1017-1398}
}

@article{JinLiu:2025,
doi = {10.1088/1361-6420/adb780},
url = {https://doi.org/10.1088/1361-6420/adb780},
year = {2025},
publisher = {IOP Publishing},
volume = {41},
number = {3},
pages = {035012},
author = {Jin, Qinian and Liu, Duo},
title = {Randomized block coordinate descent method for linear ill-posed problems},
journal = {Inverse Problems}
}

@article{JinZhouZou:2020,
    author = {Bangti Jin and Zehui Zhou and Jun Zou},
     title = {On the Convergence of Stochastic Gradient Descent for Nonlinear Ill-Posed Problems},
   journal = {SIAM J. Optim.},
      year = {2020},
    volume = {30},
    number = {2},
     pages = {1421-1450},
}

@article{EhrhardtKereta:2025,
title={A Guide to Stochastic Optimisation for Large-Scale Inverse Problems}, 
author={Matthias J. Ehrhardt and Zeljko Kereta and Jingwei Liang and Junqi Tang},
year={2025},
journal ={Inverse Prolems},
volume = {41},
number = {5},
pages = {053001, 61 pp.},
}

@article{JinZhou:2026,
doi = {10.1088/1361-6420/ae5845},
url = {https://doi.org/10.1088/1361-6420/ae5845},
year = {2026},
publisher = {IOP Publishing},
volume = {42},
number = {4},
pages = {045006},
author = {Jin, Bangti and Zhou, Zehui},
title = {On the convergence of stochastic variance reduced gradient for linear inverse problems},
journal = {Inverse Problems}
}

@incollection{JinXiaZhou:2025,
   author = {Jin, Bangti and Xia, Yuxin and Zhou, Zehui},
    title = {On the regularizing property of stochastic iterative
methods for solving inverse problems},
    booktitle = {Handbook of Numerical Analysis},
    volume = {26},
    publisher = {Elsevier, Amsterdam},
    year = {2025},
    pages = {211--272},
    issn = {1570-8659},
    doi = {https://doi.org/10.1016/bs.hna.2025.09.004},
    url = {https://www.sciencedirect.com/science/article/pii/S1570865925000043}
}

@article{P.C.Hansen2007,
    author = {P. C. Hansen},
     title = {Regularization Tools version 4.0 for {MATLAB} 7.3},
   journal = {Numer. Algorithms},
      year = {2007},
    volume = {46},
    number = {2},
     pages = {189-194},
}

@article {Morozov:1966,
    AUTHOR = {Morozov, V. A.},
     TITLE = {On the solution of functional equations by the method of
              regularization},
    journal = {Soviet Math. Dokl.},
    fjournal = {Soviet Mathematics. Doklady},
     volume = {7},
     year = {1966},
     PAGES = {414--417},
   MRCLASS = {65.35},
  MRNUMBER = {208819},
MRREVIEWER = {P.\ Hartman},
}

@article {HermanLentLutz:1978,
    AUTHOR = {Herman, Gabor T. and Lent, Arnold and Lutz, Peter H.},
     TITLE = {Relaxation method for image reconstruction},
   JOURNAL = {Comm. ACM},
  FJOURNAL = {Communications of the Association for Computing Machinery},
    VOLUME = {21},
      YEAR = {1978},
    NUMBER = {2},
     PAGES = {152--158},
      ISSN = {0001-0782},
   MRCLASS = {68A45 (92A05)},
  MRNUMBER = {0502336},
MRREVIEWER = {J. J. Buckley},
       DOI = {10.1145/359340.359351},
       URL = {https://doi.org/10.1145/359340.359351},
}

@article {Landweber:1951,
    AUTHOR = {Landweber, L.},
     TITLE = {An iteration formula for {F}redholm integral equations of the
              first kind},
   JOURNAL = {Amer. J. Math.},
  FJOURNAL = {American Journal of Mathematics},
    VOLUME = {73},
      YEAR = {1951},
     PAGES = {615--624},
      ISSN = {0002-9327},
   MRCLASS = {45.0X},
  MRNUMBER = {0043348},
MRREVIEWER = {G. E. H. Reuter},
       DOI = {10.2307/2372313},
       URL = {https://doi.org/10.2307/2372313},
}

@inproceedings{SafranShamir:2020,
   author = {Safran, Itay and Shamir, Ohad},
    title = {How good is {SGD} with random shuffling?},
     year = {2020},
   booktitle = {Proceedings of Conference on Learning Theory 2020, PMLR},
   editor = { Abernethy, Jacob and Agarwal, Shivani},
   volume = {125},
    pages = {3250--3284},
}

@inproceedings{Haochen:2019,
   author = {{HaoChen}, Jeff and Sra, Suvrit},
    title = {Random shuffling beats {SGD} after finite epochs},
  booktitle ={Proceedings of the 36th International Conference on Machine Learning, PMLR},
  editor = {Chaudhuri, Kamalika and Salakhutdinov, Ruslan},
      year = {2019},
    volume = {97},
     pages = {2624--2633},
}

@article{Ying:2018,
   author = {Ying, Bicheng and Yuan, Kun and Vlaski, Stefan and Sayed, Ali H},
    title = {Stochastic learning under random reshuffling with constant step-sizes},
  journal = {IEEE Trans. Signal Proc.},
     year = {2018},
   volume = {67},
   number = {2},
    pages = {474--489},
}

@article {HankeNeubauerScherzer:1995,
    AUTHOR = {Hanke, Martin and Neubauer, Andreas and Scherzer, Otmar},
     TITLE = {A convergence analysis of the {L}andweber iteration for
              nonlinear ill-posed problems},
   JOURNAL = {Numer. Math.},
  FJOURNAL = {Numerische Mathematik},
    VOLUME = {72},
      YEAR = {1995},
    NUMBER = {1},
     PAGES = {21--37},
      ISSN = {0029-599X},
   MRCLASS = {65J20 (47H17)},
  MRNUMBER = {1359706},
MRREVIEWER = {Hans-J\"urgen Reinhardt},
       URL = {https://doi.org/10.1007/s002110050158},
}

@book {EnglHankeNeubauer:1996,
    AUTHOR = {Engl, Heinz W. and Hanke, Martin and Neubauer, Andreas},
     TITLE = {Regularization of {I}nverse {P}roblems},
 PUBLISHER = {Kluwer, Dordrecht},
      YEAR = {1996},
     PAGES = {viii+321},
      ISBN = {0-7923-4157-0},
   MRCLASS = {65J20 (35R30 47A50 47H17)},
  MRNUMBER = {1408680},
MRREVIEWER = {Ulrich Tautenhahn},
       DOI = {10.1007/978-94-009-1740-8},
       URL = {http://dx.doi.org/10.1007/978-94-009-1740-8},
}

@article {RobbinsMonro:1951,
     AUTHOR = {Robbins, Herbert and Monro, Sutton},
     TITLE = {A stochastic approximation method},
   JOURNAL = {Ann. Math. Stat.},
  FJOURNAL = {Annals of Mathematical Statistics},
    VOLUME = {22},
      YEAR = {1951},
     PAGES = {400--407},
      ISSN = {0003-4851},
   MRCLASS = {62.0X},
  MRNUMBER = {0042668},
MRREVIEWER = {R. P. Peterson},
}

@article{ZhouDSGD:2024,
author = {Zhou, Zehui},
title = {On the convergence of a data-driven regularized stochastic gradient descent for nonlinear ill-posed problems},
journal = {SIAM J. Imaging Sci.},
volume = {18},
number = {1},
pages = {388-448},
year = {2025},
doi = {10.1137/24M1647394},
URL = {https://doi.org/10.1137/24M1647394},
eprint = {https://doi.org/10.1137/24M1647394},
}

@article {JinLu:2019,
    AUTHOR = {Jin, Bangti and Lu, Xiliang},
     TITLE = {On the regularizing property of stochastic gradient descent},
   JOURNAL = {Inverse Problems},
  FJOURNAL = {Inverse Problems. An International Journal on the Theory and
              Practice of Inverse Problems, Inverse Methods and Computerized
              Inversion of Data},
    VOLUME = {35},
      YEAR = {2019},
    NUMBER = {1},
     PAGES = {015004, 27 pp.},
      ISSN = {0266-5611,1361-6420},
   MRCLASS = {65F22 (65K05 90C90)},
  MRNUMBER = {3884601},
MRREVIEWER = {Silvia\ Gazzola},
       DOI = {10.1088/1361-6420/aaea2a},
       URL = {https://doi.org/10.1088/1361-6420/aaea2a},
}

@article {LuMathe2022,
    AUTHOR = {Lu, Shuai and Math\'e, Peter},
     TITLE = {Stochastic gradient descent for linear inverse problems in
              {H}ilbert spaces},
   JOURNAL = {Math. Comp.},
  FJOURNAL = {Mathematics of Computation},
    VOLUME = {91},
      YEAR = {2022},
    NUMBER = {336},
     PAGES = {1763--1788},
      ISSN = {0025-5718,1088-6842},
   MRCLASS = {65J20 (47A52 65F22)},
  MRNUMBER = {4435947},
MRREVIEWER = {Uno\ H\"amarik},
       DOI = {10.1090/mcom/3714},
       URL = {https://doi.org/10.1090/mcom/3714},
}

@article {JahnJin:2020,
    AUTHOR = {Jahn, Tim and Jin, Bangti},
     TITLE = {On the discrepancy principle for stochastic gradient descent},
   JOURNAL = {Inverse Problems},
  FJOURNAL = {Inverse Problems. An International Journal on the Theory and
              Practice of Inverse Problems, Inverse Methods and Computerized
              Inversion of Data},
    VOLUME = {36},
      YEAR = {2020},
    NUMBER = {9},
     PAGES = {095009, 30 pp.},
      ISSN = {0266-5611,1361-6420},
   MRCLASS = {65F22 (62L20)},
  MRNUMBER = {4149878},
MRREVIEWER = {Xiangjian\ Xu},
       DOI = {10.1088/1361-6420/abaa58},
       URL = {https://doi.org/10.1088/1361-6420/abaa58},
}

@article {JinZhouZou:2021siuq,
    AUTHOR = {Jin, Bangti and Zhou, Zehui and Zou, Jun},
     TITLE = {On the saturation phenomenon of stochastic gradient descent
              for linear inverse problems},
   JOURNAL = {SIAM/ASA J. Uncertain. Quantif.},
  FJOURNAL = {SIAM/ASA Journal on Uncertainty Quantification},
    VOLUME = {9},
      YEAR = {2021},
    NUMBER = {4},
     PAGES = {1553--1588},
      ISSN = {2166-2525},
   MRCLASS = {65J20 (65J22 90C90)},
  MRNUMBER = {4333672},
MRREVIEWER = {Jiu\ Ding},
       DOI = {10.1137/20M1374456},
       URL = {https://doi.org/10.1137/20M1374456},
}

@article {JinZhouZou:2022ip,
    AUTHOR = {Jin, Bangti and Zhou, Zehui and Zou, Jun},
     TITLE = {An analysis of stochastic variance reduced gradient for linear
              inverse problems},
   JOURNAL = {Inverse Problems},
  FJOURNAL = {Inverse Problems. An International Journal on the Theory and
              Practice of Inverse Problems, Inverse Methods and Computerized
              Inversion of Data},
    VOLUME = {38},
      YEAR = {2022},
    NUMBER = {2},
     PAGES = {025009, 34 pp.},
      ISSN = {0266-5611,1361-6420},
   MRCLASS = {65F22},
  MRNUMBER = {4374012},
       DOI = {10.1088/1361-6420/ac4428},
       URL = {https://doi.org/10.1088/1361-6420/ac4428},
}

@article {JinZhouZou:2020siopt,
    AUTHOR = {Jin, Bangti and Zhou, Zehui and Zou, Jun},
     TITLE = {On the convergence of stochastic gradient descent for
              nonlinear ill-posed problems},
   JOURNAL = {SIAM J. Optim.},
  FJOURNAL = {SIAM Journal on Optimization},
    VOLUME = {30},
      YEAR = {2020},
    NUMBER = {2},
     PAGES = {1421--1450},
      ISSN = {1052-6234,1095-7189},
   MRCLASS = {65J20 (47J06 65J22 90C90)},
  MRNUMBER = {4099813},
MRREVIEWER = {Antonio\ C. G. Leit\~ao},
       DOI = {10.1137/19M1271798},
       URL = {https://doi.org/10.1137/19M1271798},
}

@article{GaoBlumensath:2018,
    author = {Gao, Yushan and  Blumensath, Thomas},
    title = {A joint row and column action method for cone-beam computed tomography},
    journal = {IEEE Trans. Comput. Imag.},
    year = {2018}, 
    volume = {4},
    number = {4},
    pages = {599--608},
    doi = {10.1109/TCI.2018.2857446},
}
\end{document}